\documentclass[12pt]{amsart}
\usepackage{a4wide}

\usepackage{amssymb}
\usepackage{amsthm}
\usepackage{mathtools}
\usepackage{tikz-cd}
\usepackage{color}
\usepackage{upref}
\usepackage{enumitem}
\usepackage{verbatim}
\usepackage{mathrsfs}
\usepackage{hyperref}

\allowdisplaybreaks

\newcommand{\men}{Menev\c se}

\newtheorem{thm}{Theorem}[section]
\newtheorem{prop}[thm]{Proposition}
\newtheorem{cor}[thm]{Corollary}
\newtheorem{lem}[thm]{Lemma}

\theoremstyle{definition}
\newtheorem{defn}[thm]{Definition}

\newtheorem{rem}[thm]{Remark}

\numberwithin{equation}{section}

\newcommand{\secref}[1]{Section~\textup{\ref{#1}}}
\newcommand{\thmref}[1]{Theorem~\textup{\ref{#1}}}
\newcommand{\corref}[1]{Corollary~\textup{\ref{#1}}}
\newcommand{\lemref}[1]{Lemma~\textup{\ref{#1}}}

\newcommand{\defnref}[1]{Definition~\textup{\ref{#1}}}
\newcommand{\remref}[1]{Remark~\textup{\ref{#1}}}

\newcommand{\apxref}[1]{Appendix~\textup{\ref{#1}}}

\newcommand{\bb}[1]{\mathbb{#1}}
\newcommand{\cc}[1]{\mathcal{#1}}

\newcommand{\N}{\bb N}

\newcommand{\Z}{\bb Z}
\newcommand{\T}{\bb T}

\newcommand{\II}{\ensuremath{\mathscr{I}}}
\newcommand{\CC}{\ensuremath{\mathscr{C}}}
\newcommand{\DD}{\ensuremath{\mathscr{D}}}
\newcommand{\OO}{\ensuremath{\mathscr{O}}}
\newcommand{\EE}{\ensuremath{\mathscr{E}}}
\newcommand{\KK}{\ensuremath{\mathscr{K}}}

\newcommand{\LL}{\ensuremath{\mathscr{L}}}
\newcommand{\RR}{\ensuremath{\mathscr{R}}}

\newcommand{\GG}{\cc G}

\newcommand{\OX}{\OO Y_1}

\renewcommand{\epsilon}{\varepsilon}
\newcommand{\<}{\langle}
\renewcommand{\>}{\rangle}
\newcommand{\pre}[1]{{}_{#1}}

\newcommand{\inv}{^{-1}}

\newcommand{\id}{\text{id}}
\newcommand{\case}{& \text{if }}

\renewcommand{\subset}{\subseteq}

\renewcommand{\bar}{\overline}

\newcommand{\variso}{\overset{\simeq}{\longrightarrow}}
\newcommand{\xt}{\otimes}
\newcommand{\ots}{\otimes_A}
\newcommand{\otso}{\otimes_{\OX}}

\renewcommand{\)}{\textup)}

\newcommand{\cst}{\ensuremath{C^*}}
\newcommand{\csta}{\ensuremath{C^*}-algebra}

\newcommand{\V}{\overline{V}}

\newcommand{\rep}{representation}
\newcommand{\hm}{homomorphism}

\newcommand{\iso}{iso\-mor\-phism}
\newcommand{\mor}{morphism}

\newcommand{\cp}{Cuntz--Pimsner}
\newcommand{\cpa}{Cuntz--Pimsner algebra}

\newcommand{\kga}{$k$-graph algebra}

\newcommand{\alg}{algebra}
\newcommand{\kg}{$k$-graph}

\newcommand{\corr}{correspondence}
\newcommand{\ps}{product system}

\newcommand{\hex}{hexagonal relations}
\newcommand{\moncat}{monoidal category}

\newcommand{\gensys}{generating system}

\newcommand{\aut}{\operatorname{Aut}}
\newcommand{\auto}{automorphism}

\newcommand{\midtext}[1]{\quad\text{#1}\quad}
\newcommand{\righttext}[1]{\quad\text{#1 }}

\newcommand{\algcor}{\ensuremath{\mathbf{C^*alg_{cor}}}}
\newcommand{\cscor}{\ensuremath{\mathbf{C^*cor_{pair}}}}

\begin{document}

\title[k-graph algebras from the bottom up]{k-graph algebras are iterated Cuntz--Pimsner algebras\\
--- from the bottom up}

\author[Deaconu]{Valentin Deaconu}
\address{Department of Mathematics and Statistics
\\University of Nevada
\\Reno, Nevada 89557}
\email{vdeaconu@unr.edu}

\author[Ery\"uzl\"u Paulovicks]{\men\ Ery\"uzl\"u Paulovicks}
\address{Department of Mathematical Sciences
\\New Mexico State University
\\Las Cruces, New Mexico 88003}
\email{Menevse.Paulovicks@gmail.com}

\author[Kaliszewski]{S. Kaliszewski}
\address{School of Mathematical and Statistical Sciences
\\Arizona State University
\\Tempe, Arizona 85287}
\email{kaliszewski@asu.edu}

\author[Quigg]{John Quigg}
\address{School of Mathematical and Statistical Sciences
\\Arizona State University
\\Tempe, Arizona 85287}
\email{quigg@asu.edu}

\date{\today}


\dedicatory{We dedicate this paper to Gene Abrams on the occasion of his 70th birthday.}

\date{\today}

\begin{abstract}
We introduce a new method of expressing a $k$-graph $C^*$-algebra as a Cuntz--Pimsner algebra. Kumjian, Pask, and Sims have done this directly, using a linking algebra approach and a $(k-1)$-graph algebra. This can be iterated downward. Our process, on the other hand, starts at the bottom, with Pimsner's theorem for graph algebras, and iterates upward.
We actually work with product systems over $\N^k$, and the result for $k$-graphs is a special case.
Our iteration step involves a ``decategorization'' of a recent theorem showing that the Cuntz--Pimsner construction is functorial at the level of ``enchilada categories''.
\end{abstract}

\maketitle

\section{Introduction}\label{sec:intro}

This paper combines
the following topics:
directed graph \csta s
and their higher-rank generalizations,
\cst-\corr s
and product systems of these,
and their \cp\
\alg s.
And a bit of category theory ---
at least in spirit if not in body.

It is tempting to imagine that one of Pimsner's primary motivations for inventing (what are now called) \cpa s (see \cite{pimsner}) was his realization that the \csta\ of a directed graph is naturally isomorphic to the \cpa\ of an associated \cst--\corr\ formed from the edges and vertices.
The \kg s introduced by Kumjian and Pask in \cite{kp} are rank-$k$ generalizations of directed graphs, and have \csta s formed by similar generators and relations. So it's tempting to ask whether these $k$-graph \alg s are also \cp\ \alg s. It's certainly too na\"ive to try for a \corr\ over functions on the vertices --- any such \corr\ would be isomorphic to one associated to a graph, as shown in \cite{kpq}.

But this is not a serious obstacle;
Raeburn and Sims showed in \cite{raesim} that there is a product system, over $\N^k$, of \corr s, and Fowler showed us in \cite{fow} how to construct an associated \csta.
This \ps\ is constructed from the edges in the $k$ dimensions.

It would still be good to have a \corr\ whose \cpa\ is isomorphic to the $k$-graph \alg.
Kumjian, Pask, and Sims (see \cite{morph}) gave one solution to this problem: use the edges in dimension $k$ to give a \corr\ over the \csta\ of the $(k-1)$-graph formed by the edges in the first $k-1$ dimensions.
They then used a linking algebra to prove that the associated \cpa\ is isomorphic to the $k$-graph \alg. This process could clearly be iterated downward $k$ times.
Something along the
same lines (but with fewer restrictions on the \kg s) is done in \cite{fletcher}.

Our purpose in this paper is to introduce a new method, which proceeds from the bottom up rather than the top down. We begin with Pimsner's theorem to get the \alg\ of the graph in dimension 1 as a \cpa. To continue, we apply a technique derived from a functor developed by the second author (see \cite{functor}), using an inductive process $k$ times with the aim of arriving at the \kga.
There is an extra step required each time: the construction derived from the functor produces a \corr\ that is not directly applicable to \kg s. To fix this, we must apply a tool introduced by the first author (see \cite{deaconu}) to reformulate the \corr\ in terms of \ps s. Actually, in its original form, this tool is applicable only for 2-graphs. To complete our iterative process we generalize this tool to \kga s.

There is a subtlety here: what our iterative process produces are \csta s of product systems over powers of $\N$.
To get the higher-rank graph \alg s
we apply the Raeburn--Sims theorem.
However, the initial output of the process gives us something more general.
In fact, our inductive process can be regarded as producing the \alg\ of a given \ps\ over $\N^k$ in $k$ steps.
The product systems associated to $k$-graphs are really only special cases.
In fact, this can be seen already when $k=2$:
the third and fourth authors, in collaboration with Patani (see \cite{kpq2}), have shown that
there is an obstruction preventing arbitrary product systems over $\N^2$ from being 2-graph product systems.
In this way, our method is somewhat more general than \kg s.

In the above remarks we glossed over the connection with category theory. The functor introduced by the second author could, in principle, be used to give a categorical procedure. However, due to the nature of the relevant categories, this would involve a lot of passing back and forth between \corr s and their \iso\ classes. So in the end we reluctantly decided to forego the functorial approach. Nevertheless, the iterative process itself is most clearly formulated categorically, and in \secref{abstract} we present the abstract procedure;
however, we stress that this categorical formulation is not used in the main body of our paper, and is included purely for illustrative purposes.

In \secref{prelim}
we state our conventions regarding \kg s, \corr s,
product systems, and the associated \csta s.
We emphasize that we impose fairly stringent conditions on our \kg s and \corr s: we require them to be nondegenerate, and also
\emph{regular} in the sense
that the \kg s must be row-finite and source-free,
and the \corr s must have left module homomorphism
that is injective into the compacts.
Presumably these assumptions could be relaxed,
but we believe that the clearest exposition of our methods requires the regularity assumption.

\secref{abstract} contains (as we mentioned above) a quite brief illustration of our method in the abstract realm of category theory, which we again emphasize is not used in the formulation of our main results.
In \secref{product} we give a precise description of our iterative process, using what we call 
\emph{\gensys s},
by which we mean that we work exclusively with the \corr s associated to the generators $e_i$ of $\N^k$.
The main technical tool we need involves the
``tensor groupoids'' of Fowler--Sims, although we use the more traditional terminology ``monoidal categories''.
Our main result is \corref{main}, showing that the \kga\ is a \cpa. As mentioned above, we deduce this as a special case of the general version for \ps s (\thmref{for product}).
The details of the proof of this proposition are somewhat messy, and to avoid interrupting the flow we relegate many of them to \secref{detail}.
We further relegate the proofs of a few technical lemmas to an appendix.

A recent preprint
\cite{huengsim}
uses two iterations of a version of this Cuntz--Pimsner procedure (they refer to it as the ``Deaconu--Fletcher constructions'') to prove K-theory results for 2-graphs.

A concluding remark:
given the wide interest in Leavitt path algebras,
which are purely algebraic versions of graph \csta s,
it seems natural to try to do something along the lines of the present paper
for Kumjian--Pask algebras \cite{claflyhue},
which are higher-rank versions of Leavitt path algebras.
We leave this to the experts in the Leavitt path algebra world.


\section{Preliminaries}\label{prelim}

Throughout, all $A-B$ \corr s
$\pre AX_B$
will be \emph{nondegenerate} in the sense that $AX=X$
and \emph{regular} in the sense that the left-module \hm\ 
\mbox{$\varphi=\varphi_X:A\to \LL(X)$}
maps injectively into $\KK(X)$. 

A \emph{representation} $(\pi,t)$ of $\pre AX_A$ on a $C^*$-algebra $B$ consists of a $*-$homomorphism $\pi: A\rightarrow B$ and a linear map $t: X\rightarrow B $ such that 
\[
\pi(a)t(x)=t(\varphi_X(a)(x)) \midtext{and} t(x)^* t(y)=\pi(\<x,y\>_A), 
\]
for $a\in A$ and $x, y\in X.$  For each representation $(\pi,t)$ of $\pre AX_A$ on $B$,  there exists a homomorphism $t^{(1)}: \KK(X)\rightarrow B$ such that 
\[
t^{(1)} (\theta_{x,y})=t(x){t(y)}^*
\]
for $x,y\in X. $ The representation $(\pi,t)$ is called \emph{injective} if $\pi$ is injective, in which case $t$ is an isometry and $t^{(1)}$ is injective. We denote the $C^*$-algebra generated by the images of $\pi$ and $t$ in $B$ by $C^*(\pi, t).$  Since we assume all \corr s are regular,
a \rep\ $(\pi,t)$ of $\pre AX_A$ in a \csta\ $B$ is covariant iff $ \pi=t^{(1)}\circ\varphi_X.$
The \emph{Cuntz-Pimsner algebra} $\OO X$ of $\pre AX_A$ is the $C^*$-algebra generated by the universal covariant representation, which we denote by $(\pi_X, t_X)$ throughout the paper\footnote{where we adopt Katsura's notational convention}. This representation is injective, and admits a gauge action.  

A \emph{$C^*$-correspondence automorphism} of $\pre AX_A$ is a pair  $(i_A, i_X )$ consisting of a bijective linear map $i_X: X\rightarrow X$, and an isomorphism $i_A: A\rightarrow A$ satisfying
\begin{enumerate}
\item $i_X(a\cdot x)=i_A(a)\cdot i_X(x)$,
\item $i_A( \<x,z\>_A ) = \<i_X(x), i_X(z)\>_{A}$,
\end{enumerate}
for all $a\in A$, and $x,z\in X.$


We recall the \corr\ category of \cite{functor}\footnote{although we impose stronger regularity conditions}.
The \emph{enchilada category} \algcor\ of \csta s has
\begin{itemize}
\item
Objects: \csta s $A$;

\item
Morphisms: $[X]:A\to B$ is the \iso\ class
of a (nondegenerate regular) $A-B$ \corr\ $X$.
\end{itemize}
The \emph{enchilada category} \cscor\ of \corr s has
\begin{itemize}
\item
Objects: regular \corr s $\pre AX_A$;

\item
Morphisms: $[M,U]:\pre AX_A\to \pre BY_B$
is the \iso\ class of a pair $(M,U)$,
where $M$ is a regular $A-B$ \corr\
making the diagram
\[
\begin{tikzcd}
A \arrow[r,"X"] \arrow[d,"M"']
&A \arrow[d,"M"]
\\
B \arrow[r,"Y"']
&B
\end{tikzcd}
\]
commute,
$U:X\ots M\variso M\xt_BY$ is an isomorphism,
and $(M,U)$ is \emph{isomorphic} to another such pair $(M',U')$ if
there is an \iso\ $\xi:M\variso M'$ making the diagram
\[
\begin{tikzcd}
X\ots M \arrow[r,"U"] \arrow[d,"1\xt\xi"']
&M\xt_BY \arrow[d,"\xi\xt 1"]
\\
X\ots M' \arrow[r,"U'"']
&M'\xt_BY
\end{tikzcd}
\]
commute.
\end{itemize}
The composition of \mor s
$[M,U]:\pre AX_A\to \pre BY_B$
and
$[N,V]:\pre BY_B\to \pre CZ_C$
is the \iso\ class
\[
[N,V]\circ [M,U]:=[M\xt_BN,(1\xt V)(U\xt 1)].
\]

For our gauge action arguments the essential proposition we are using is derived from \cite[Proposition~2.9]{robszy}:
\begin{prop}\label{comp}
For a regular \cst-correspondence $\pre AX_A$, let  $(i_A, i_X): \pre AX_A\rightarrow \pre AX_A$ be a \cst-correspondence \auto. Then there exists an automorphism $i_{X*}: \OO X\rightarrow \OO X$ such that 
\[
i_{X*}\circ t_X=t_X\circ i_X \text{ and } i_{X*}\circ\pi_X=\pi_X\circ i_A.
\]
\end{prop}

We will also need the following:
let $\pre AY_A$ be a \cst-correspondence, and let $\gamma$ denote the gauge action on $\OO Y$.  For any $n\in \Z$ consider the subspace
\[ (\OO Y)^n:=\{T\in \OO Y: \gamma_z(T)=z^n(T), \text{ for all } z\in\T\}.\] Then observe that we have 
\begin{align*}
\OO Y&=\overline{\text{span}}\{t_{Y}^n(y_n)t_{Y}^m(y_m)^*: y_n\in Y^{\otimes n}, y_m\in Y^{\otimes m}, n,m\geq 0\}\\
&=\overline{\text{span}}\{T_s\in(\OO Y)^s: s\in \Z\}.
\end{align*}

We adapt to our needs a bit of the theory of multiplier bimodules from \cite[Sections~1.2--1.3]{enchilada}.
Given an $A$-\corr\ $X$ the \emph{multiplier bimodule} $M(X)$ is $\LL(A,X)$ with natural $M(A)$~\corr\ operations and strict topology,
and is the strict completion of $X$.
If $X$ and $Y$ are $A$-\corr s, any nondegenerate\footnote{where \emph{nondegenerate} here means $\Phi(X)A=Y$}
 \corr\ \hm\ $\Phi:X\to M(Y)$ extends uniquely to a strictly continuous \corr\ \hm\ $\bar\phi:M(X)\to M(Y)$ (\cite[Theorem~1.30]{enchilada}).
Given $A$-\corr s $X$ and $Y$,
by \cite[Lemma~1.32]{enchilada}
there is a natural embedding
\[
M(X)\xt_{M(A)}M(Y)\hookrightarrow M(X\ots Y).
\]
If we also have $A$-\corr s $Z,W$ and \corr\ \hm s $\phi:X\to M(Z),\sigma:Y\to M(W)$,
by \cite[Proposition~1.34]{enchilada}
there is a natural \corr\ \hm\
\[
\phi\xt\sigma:X\ots Y\to M(Z\ots  W),
\]
which is nondegenerate if $\phi$ and $\sigma$ are.

We will work frequently with \ps s over the commutative monoid $\N^k$,
and we list the axioms in the relevant form:
a \emph{product system} over $\N^k$ of $A$-\corr s is a 
semigroup $X=\bigsqcup_{s\in \N^k}X_s$,
where each $X_s$ is a
nondegenerate
regular
$A$-\corr, such that
\begin{itemize}
    \item $X_0=\pre AA_A$ with the canonical $A$-\corr\ structure, and

    \item the multiplication in $X$ implements \iso s
$X_s\ots X_t\simeq X_{s+t}$ for all 
$s,t\in\N^k$.
\end{itemize}

Let 
$B$ be a \csta.
A \emph{\rep} $\psi:X\to B$ of a \ps\ $X$
comprises \rep s
$\psi_s:X_s\to B$ such that
\[
\psi_s(x)\psi_t(y)=\psi_{s+t}(xy)\righttext{for all}x\in X_s,y\in X_t,
\]
and
$\psi$ is \emph{\cp\ covariant} if each $\psi_s$ is.
The \emph{Cuntz--Pimsner algebra} $\OO X$ is the  $C^*$-algebra generated by the universal covariant representation $\psi$ of $X$. It has a natural gauge action $\gamma$ of $\T^k$ such that
$\gamma_{z}(\psi_s(x))=z^s \psi_s(x),$ for $z\in \N^k$ and $x\in X_s.$ 
\cite[Section~3.3]{dkps}.
The Gauge-Invariant Uniqueness theorem for \ps s is:

\begin{lem}[{\cite[Lemma~3.3.2]{dkps}}]\label{gauge-lem}
Let $X=\bigsqcup_{s\in\N^k}X_s$ be a \ps\ of regular $C^*$-correspondences. Let $B$ be a $C^*$-algebra equipped with a
strongly continuous action $\beta : \T^k\to \aut B$, and let $\sigma=\{\sigma_s\}_{s\in \N^k}$ be a covariant representation of $X$ in B. 
Suppose that the following conditions are satisfied.
\begin{enumerate}[label=\textup{(\arabic*)}]
\item For all $z\in\T^k, s\in\N^k, x\in X_s$ we have
$\beta_z(\sigma_s(x))=z^s\sigma_s(x)$.

\item The homomorphism $\sigma_0: A\to B$ is injective.
\end{enumerate}
Then the induced \hm\ $\sigma_{*}: \OO X\to B $ is injective. 
\end{lem}

By
our blanket regularity assumption,
$X$ is automatically compactly aligned (by \cite[Proposition~5.8]{fow}).
Moreover, every \cp\ covariant \rep\ is automatically Nica covariant (see \cite[Proposition~5.4]{fow} or \cite[Lemma~6.1]{raesim}),
and even Cuntz--Nica--Pimsner covariant \cite[Corollary~5.2]{simyee}.

It will be convenient for our purposes to have available the more abstract notion of \ps\ from \cite{fowsim}\footnote{except that we adopt the conventions from, e.g., \cite[Section~VII.1]{maclane}}:
a \emph{\moncat}
is a category $\GG$ equipped with a bifunctor $\xt:\GG\times\GG\to \GG$ and an ``identity'' object $1_\GG$ 
such that, up to natural isomorphism, the product $\xt$ is associative and $1_\GG$ acts like an identity for $\xt$.

For our purposes the most important example of a \moncat\ consists of $A$-\corr s\footnote{for some fixed \csta\ $A$, and as always we tacitly assume the \corr s are nondegenerate and regular}, with multiplication given by balanced tensor product, and identity object $\pre AA_A$,
and in which in which a morphism $\theta:X\to Y$ is an \iso\ of $A$-\corr s.

\begin{defn}
A \emph{product system} over $\N^k$ in a \moncat\ $\GG$ is a
family $\{X_s\}_{s\in \N^k}$ of objects in $\GG$
such that
\begin{itemize}
\item
$X_0=1_\GG$

\item
$X_s\xt X_t\simeq X_{s+t}$ for all $s,t\in \N^k$,
\end{itemize}
where all \iso s are natural and associative\footnote{up to natural isomorphism}
%
\end{defn}

Of particular interest to us are \ps s of $A$-\corr s,
with
$X_0=\pre AA_A$.

We want to construct \rep s of $X$ from certain tuples of \rep s of some of the $A$-\corr s $Y_1,\dots,Y_k$,
and we start with a concept that makes sense in any \moncat:

\begin{defn}\label{gensys}
    A \emph{\gensys} 
    in a \moncat\
    is a pair $(Y,\theta)$, where
    $Y=\{Y_i:i=1,\dots,k\}$ is a collection of objects in $\GG$
    and
    $\theta=\{\theta_{ij}:i,j=1,\dots,k\}$ is a collection of isomorphisms
    \[
    \theta_{ij}:Y_i\xt Y_j\variso Y_j\xt Y_i\righttext{for all}i,j,
    \]   
    such that
    \begin{itemize}
    \item
    $\theta_{ii}=\id$;
        
    \item
    $\theta$ is \emph{symmetric} in the sense that $\theta_{ij}=\theta_{ji}\inv$;
    
    \item
    the \emph{\hex} hold: for all $i,j,\ell$
    the diagram
    \[
    \begin{tikzcd}
    &Y_j\xt Y_i\xt Y_\ell \arrow[dr,"1\xt \theta_{i\ell}"]
    \\
    Y_i\xt Y_j\xt Y_\ell \arrow[ur,"\theta_{ij}\xt 1"] \arrow[d,"1\xt \theta_{j\ell}"']
    &&Y_j\xt Y_\ell\xt Y_i \arrow[d,"\theta_{j\ell}\xt 1"]
    \\
    Y_i\xt Y_\ell\xt Y_j \arrow[dr,"\theta_{i\ell}\xt 1"']
    && Y_\ell\xt Y_j\xt Y_i
    \\
    &Y_\ell\xt Y_i\xt Y_j \arrow[ur,"1\xt \theta_{ij}"']
    \end{tikzcd}
    \]
    commutes.
    \end{itemize}
\end{defn}

It will be convenient for us to extract from \cite[Theorems~2.1, 2.2]{fowsim} a single result:

\begin{thm}[{\cite{fowsim}}]\label{fs}
If $(Y,\theta)$ is a \gensys\ in a \moncat\ $\GG$,
then there is a \ps\ $X$ over $\N^k$ in $\GG$ such that
\[
X_{e_i}=Y_i\righttext{for}i=1,\dots,n.
\]
Moreover, $X$ is essentially unique \(i.e., unique up to isomorphism\), and every \ps\ over $\N^k$ arises in this way\footnote{up to isomorphism, of course}.
\end{thm}

Now we return to \rep s in a \csta\ $B$ of a \ps\ $X$ of $A$-\corr s.

\begin{defn}
Let $\GG$ be the \moncat\
of $A$-\corr s, and let $\pi$ be a \hm\ of $A$ into a \csta\
$B$.
We define a new \moncat\ $\GG^B$
whose objects are pairs $(\sigma,X)$,
where 
$X$ is an $A$-\corr\ 
and $\sigma$ is a \rep\ of $X$ in $B$
with coefficient \hm\ $\pi$ --- i.e., $\sigma(axc)=\pi(a)\sigma(x)\pi(c)$
for $a,c\in A,x\in X$.
A morphism 
$\zeta:(\sigma,X)\to (\omega,Y)$ in $\GG^B$ 
is a \corr\ \iso\  $\zeta$
making the diagram

\[
\begin{tikzcd}
    X \arrow[r,"\zeta","\simeq"'] \arrow[d,"\sigma"']
    &Y \arrow[dl,"\omega"]
    \\
    B
\end{tikzcd}
\] commute.
We define a multiplication in $\GG^B$ by
\[
(\sigma,X)\xt (\omega,Y)=(\sigma\omega,X\ots Y),
\]
where $\sigma\omega$ is the \rep\ of $X\ots Y$ in $B$ given on elementary tensors by
\[
\sigma\omega(x\ots y)=\sigma(x)\omega(y)
\righttext{for}x\in X,y\in Y,
\]
and we define an identity object by
\[
1_{\GG^B}=(\pi,\pre AA_A).
\]
It's a routine exercise to verify the axioms of a monoidal category.
\end{defn}

We will denote a \gensys\ in the \moncat\ $\GG^B$ by 
$(\sigma,Y,\theta)$,
where now we have
    \[
    \theta_{ij}:(\sigma_i\sigma_j,Y_i\ots Y_j)\variso (\sigma_j\sigma_i,Y_j\ots Y_i)\righttext{for all}i,j.
    \]
Note that $(Y,\theta)$ is a \gensys\ in $\GG$, 
and the diagram
\[
\begin{tikzcd}
    Y_i\ots Y_j \arrow[r,"\theta_{ij}","\simeq"'] \arrow[d,"\sigma_i\sigma_j"']
    &Y_j\ots Y_i \arrow[dl,"\sigma_j\sigma_i"]
    \\
    B
\end{tikzcd}
\] commutes for all $i,j$.

Similarly, we denote a \ps\ over $\N^k$ in $\GG^B$ by $(\psi,X)$, where we are given \iso s
\[
(\psi_s,X_s)\xt (\psi_t,X_t)\variso
(\psi_{s+t},X_{s+t})
\righttext{for all}s,t\in \N^k.
\]

We will always want \rep s of \corr s to be \cp\ covariant.
For this it will be convenient to have a monoidal subcategory of $\GG^B$:

\begin{defn}
With the above notation, we define a full subcategory $\GG^B_c$ of $\GG^B$
by taking only the objects $(\sigma,X)$ in which the representation $\sigma$ is \cp\ covariant.
\end{defn}

We want to prove that $\GG^B_c$ is a monoidal subcategory. For this we will need the following Lemma:

\begin{lem}\label{compact}
Take two regular correspondences $X$ and $Y$ over $A$. For $a\in A$, let $\varphi_X(a)=\lim_{s\rightarrow\infty}\sum_{r=1}^{N_s}\theta_{x_r^sa_r^s,\bar{x}_r^s}$, where $x_r^s,\bar{x}_r^s \in X,$ and $a_r^s\in A.$ Moreover, for each $s,r$ let \[\varphi_Y(a_r^s)=\lim_{t\rightarrow\infty}\sum_{j=1}^{N_{srt}}\theta_{y_j^{srt},k_j^{srt}}\] for $y_j^{srt}, k_j^{srt}\in Y.$ Then we have the equality 
\[\varphi_X(a)\otimes 1_Y=\lim_{s\rightarrow\infty}\sum_{r=1}^{N_s}\left[\lim_{t\rightarrow\infty}\sum_{j=1}^{N_{srt}}\theta_{x_r^s\ots y_j^{srt},\bar{x}_r^s\ots k_j^{srt}}\right]\]
\end{lem}

\begin{proof}
\begin{align*}
\varphi_X(a)\otimes 1_Y (z\ots y)&=\lim_{s\rightarrow\infty}\sum_{r=1}^{N_s}x_r^sa_r^s\<\bar{x}_r^s,z\>_A\ots y\\
&=\lim_{s\to\infty}\sum_{r=1}^{N_s}x_r^s\ots\varphi_Y(a_r^s)\<\bar{x}_r^s,z\>_A\cdot y \\
&=\lim_{s\rightarrow\infty}\sum_{r=1}^{N_s}x_r^s\ots\left[\lim_{t\rightarrow\infty}\sum_{j=1}^{N_{srt}}\theta_{y_j^{srt},k_j^{srt}}\<\bar{x}_r^s,z\>_A\cdot y\right]\\
&=\lim_{s\rightarrow\infty}\sum_{r=1}^{N_s}\left[\lim_{t\rightarrow\infty}\sum_{j=1}^{N_{srt}}\theta_{x_r^s\ots y_j^{srt},\bar{x}_r^s\ots k_j^{srt}}\right](z\ots y),
\end{align*}
as desired. 
\end{proof}

\begin{lem}\label{subcat}
With the above notation, $\GG^B_c$ is a monoidal subcategory.
\end{lem}

\begin{proof}
We must prove that $\GG^B_c$ is closed under products:
for objects $(\sigma,X),(\omega,Y)$,
we must show that if $\sigma$ and $\omega$ are both \cp\ covariant
then so is the product $(\sigma\omega,X\ots Y)$.
Write $\pi$ for the common \rep\ of $A$. Then for $a\in A$, by using  Lemma~\ref{compact}, we get 
\begin{align*}
(\sigma\omega)^{(1)}\circ \varphi_{X\ots Y}(a)
&=(\sigma\omega)^{(1)}(\varphi_X(a)\ots 1_Y)\\
&=\lim_{s\rightarrow\infty}\sum_{r=1}^{N_s}\left[\lim_{t\rightarrow\infty}\sum_{j=1}^{N_{srt}}\sigma(x_r^s)\omega(y_j^{srt})\omega(k_j^{srt})^*\sigma(\bar{x}_r^s)^*\right]\\
&=\lim_{s\rightarrow\infty}\sum_{r=1}^{N_s}\sigma(x_r^s)\left[\omega^{(1)}\circ \lim_{t\rightarrow\infty}\sum_{j=1}^{N_{srt}}\theta_{y_{j}^{srt},k_j^{srt}}\right]\sigma(\bar{x}_r^s)^*\\
&=\lim_{s\rightarrow\infty}\sum_{r=1}^{N_s}\sigma(x_r^s)\pi(a_r^s)\sigma(\bar{x}_r^s)^*\\
&=\sigma^{(1)}\circ \varphi_X(a)\\
&=\pi(a). 
\end{align*}
The last step follows from the fact that $(\pi,\sigma)$ is covariant.  
\end{proof}

The following lemma is routine, and is the whole reason for introducing the auxiliary \moncat\ $\GG^B$:
\begin{lem}\label{ps iff rep}
Let $X$ be a \ps\ over $\N^k$
in the \moncat\
$\GG$ of $A$-\corr s,
and let $B$ be a \csta.
Then $(\psi,X)$ is a \ps\ in $\GG^B$
iff
$\psi$ is a \rep\ of $X$ in $B$.
\end{lem}

Here is our desired tool for generating \rep s of \ps s over $\N^k$:

\begin{lem}\label{replem}
Suppose that $X$ is a \ps\ over $\N^k$ of $A$-\corr s,
with \gensys\ $(Y,\theta)$.
Additionally, suppose that
we are given a \csta\ $B$ and
\rep s $(\sigma_0,\sigma_i)$ of $Y_i$ in $B$ for $i=1,\dots,k$.

Then there is a 
representation $\psi$ of $X$ in $B$ such that
$\psi_{e_i}=\sigma_i$ for all $i$, and $\psi_0=\sigma_0$,
iff
\begin{equation}\label{intertwine}
\sigma_i\sigma_j=\sigma_j\sigma_i\circ \theta_{ij}\righttext{for all}i,j.
\end{equation}
Moreover, $\psi$ is essentially unique,
and is \cp\ covariant if every $\sigma_i$ is.
\end{lem}

\begin{proof}
First assume \eqref{intertwine}.
To get a \rep\ $\psi$ of $X$ in $B$,
by \lemref{ps iff rep} it suffices to show that $(\sigma,Y,\theta)$ is a \gensys\ in $\GG^B$.
The only nonobvious property is the \hex.
For this purpose, consider the diagram
\[
\begin{tikzcd}
    &Y_j\ots Y_i\ots Y_\ell
    \arrow[dr,"1\ots \theta_{i,\ell}"]
    \arrow[dd,"\sigma_j\sigma_i\sigma_\ell"]
    \\
    Y_i\ots Y_j\ots Y_\ell
    \arrow[ur,"\theta_{ij}\ots 1"]
    \arrow[dd,"1\ots \theta_{j\ell}"']
    \arrow[dr,"\sigma_i\sigma_j\sigma_\ell"']
    &&Y_j\ots Y_\ell\ots  Y_i
    \arrow[dd,"\theta_{j\ell}\ots 1"]
    \arrow[dl,"\sigma_j\sigma_\ell \sigma_i"]
    \\
    &B
    \\
    Y_i\ots Y_\ell\ots  Y_j
    \arrow[dr,"\theta_{i,\ell}\ots 1"']
    \arrow[ur,"\sigma_i\sigma_\ell \sigma_j"]
    &&Y_\ell\ots Y_j\ots Y_i
    \arrow[ul,"\sigma_\ell \sigma_j\sigma_i"']
    \\
    &Y_\ell\ots  Y_i\ots Y_j.
    \arrow[ur,"1\ots \theta_{ij}"']
    \arrow[uu,"\sigma_\ell \sigma_i\sigma_j"]
\end{tikzcd}
\]
The $\theta$'s are all \corr\ \iso s,
the triple products of $\sigma$'s
make sense in the standard manner,
and it follows from \eqref{intertwine} that
the diagram commutes.
Thus $(\sigma,Y,\theta)$ is a \gensys\ in $\GG^B$.
The essential uniqueness of the associated \ps\ $(\psi,X)$
follows from \thmref{fs}.

For the covariance, assuming that the \rep s $\sigma_i$ of the $Y_i$ are covariant,
we need to know that the \rep\ $\psi$ of $X$ is also covariant.
This will follow from another application of \thmref{fs}
by restricting to the monoidal subcategory $\GG^B_c$.

For the converse direction, suppose that we have a \rep\ $\psi$ of $X$ in $B$
such that
$\psi_{e_i}=\sigma_i$ for all $i$
and $\psi_0=\sigma_0$.
Then as above we know that $(\psi,X)$ is a \ps\ over $\N^k$ in $\GG^B$ such that
$(\psi_{e_i},X_{e_i})=(\sigma_i,Y_i)$
for all $i$.
It follows from \thmref{fs} that
$X$ is essentially unique,
and then a straightforward induction argument
shows that $\psi$ is (essentially) uniquely determined by $\sigma$.
\end{proof}

The
functor $\EE:\cscor\to \algcor$ from \cite{functor}
takes an object $X$ to the \cpa\ $\OO X$,
and a \mor\ $M:\pre AX_A\to \pre BY_B$
to a \mor\ $\EE M:\OO X\to \OO Y$.
The explicit formula is
\[
\EE M=[M\ots \OO X],
\]
with an explicit construction of the $\OO X-\OO Y$ \corr\ structure.

Since we are interested in working with \corr s and \kg s,
we need a formal connection:
in \cite[Example~3.1 and Corollary~4.4]{raesim},
Raeburn and Sims
relate \kga s to \cpa s of \ps s:
if $\Lambda$ is a regular \kg\ (i.e., row-finite with  no sources),
and if for each $n\in\N^k$ we let $X_n$ be the graph \corr\
associated to the directed graph with edges $\Lambda^n$ and vertices $\Lambda^0$,
then the \corr s $Y_i:=X_{e_i}$ (where $e_1,\dots,e_k$ are the standard generators of $\N^k$)
generate a \ps\ $\ast_{i=1}^k Y_i$, and
\[
\OO \Lambda
\simeq \OO (\ast_1^kY_j).
\]
Then to make the connection between \ps s and our iterative use of the second author's functor requires a result (\thmref{for product}) that grew out of \cite[Lemma~4.2 and Remark~4.3]{deaconu}.

\section{Abstract iterative process}\label{abstract}

To introduce our method cleanly, in this section we present an abstract categorical approach. However, we must clarify that, as we indicate in \secref{product}, unfortunately we will not be using the categorical approach when we develop our method for product systems.

Starting with any category $\CC$,
a standard construction is the \emph{arrow category},
written $\CC^\to$,
which has
\begin{itemize}
\item
Objects: arrows $f:a\to b$ in $\CC$;

\item
Morphisms: commuting squares in $\CC$:
\[
\begin{tikzcd}
a \arrow[r,"f"] \arrow[d,"h"']
&b \arrow[d,"h'"]
\\
c \arrow[r,"g"']
&d
\end{tikzcd}
\]
where we regard the pair $(h,h')$ as a morphism in $\CC^\to$ from $f$ to $g$.
\end{itemize}

Our interest is in what we call the \emph{intertwiner \textup (sub\textup )category},
which we write as $\II\CC$. The only difference is that we take the $f$ and $g$ to be endo\mor s, and we take $h=h'$:
\[
\begin{tikzcd}
a \arrow[r,"f"] \arrow[d,"h"']
&a \arrow[d,"h"]
\\
b \arrow[r,"g"']
&b.
\end{tikzcd}
\]
So now we regard this as a morphism $h:f\to g$ in $\II\CC$.

We assume that we have a functor $E:\II\CC\to \CC$,
but (for secret yet obvious reasons) we denote what $E$ does to objects by $\OO$:
\begin{itemize}
\item
Objects: $f\mapsto \OO f$

\item
Morphisms: $(h:f\to g) \mapsto (E h:\OO f\to \OO g)$.
\end{itemize}

Now suppose that we have 
an object $a$ of $\CC$
and
a commuting set of endo\mor s $S\subset \CC(a)$.
Choose $f_1\in S$.
Then we can
regard $S$ as a 
commuting subset of
of $\II\CC(f_1)$.
This works, because for any $f\in S$ we have a commuting square
\[
\begin{tikzcd}
a \arrow[r,"f_1"] \arrow[d,"f"']
&a \arrow[d,"f"]
\\
a \arrow[r,"f_1"']
&a.
\end{tikzcd}
\]
We apply the functor $E$ to get a 
commuting subset
$E S$ of $\CC(\OO f_1)$.

We can iterate this process:
in general we have
\[
E^jS\subset \CC(\OO E^{j-1}f_j)
\]
and then choosing $f_{j+1}\in S$ and  regarding
\[
E^jS\subset \II\CC(E^jf_{j+1}),
\]
we get
\[
E^{j+1}S\subset \CC(\OO E^jf_{j+1}).
\]

\section{Main result}\label{product}

The context for our iterative process involves \ps s.
We could let $\CC=\algcor$ and 
$\DD=\cscor$\footnote{Remember that we are restricting to regular \corr s.}
and 
use the
functor $\EE:\DD\to\CC$ from \cite{functor}.
At first glance it looks like we're ready to perform our iterative process.
But there 
would
be numerous roadblocks along the way.
For example, $\DD$ is \emph{not} the intertwiner category $\II\CC$ discussed in \secref{abstract};
the objects of the latter would be \iso\ classes of \corr s,
whereas the objects of $\DD$ are \corr s themselves.
Navigating these roadblocks 
would
involve a lot of gymnastics passing back and forth between \corr s and their \iso\ classes.
We have decided to eschew this,
at the expense of losing the category theory itself.
Instead, we refrain from actually using the categories $\CC$ and $\DD$, instead working directly with actual \corr s rather than \iso\ classes.

More precisely, we work with 
\gensys s $(Y,\theta)$ (see \defnref{gensys}).
Fowler and Sims \cite{fowsim} proved that in fact the existence of such a
\gensys\
is
necessary and sufficient for having an actual \ps\ over $\N^k$
(as we explained in \lemref{replem}),
and we exploit this tool extensively.

Let 
$(Y,\theta)$ be a \gensys\ of $A$-\corr s.
Let $Y_1=X_{e_1}$.Then for each 
$i=1,\dots,k$
the \iso\ 
\[\theta_{1s}:Y_1\ots Y_i\to Y_i\ots Y_1\]
gives rise to the
$\OO Y_1$-\corr\ $\EE Y_i=Y_i\ots \OO Y_1$.
For our iterative process, we now
define 
a \gensys\ $(\EE Y,\RR \theta)$,
with
\begin{itemize}
\item
$\EE Y=\{\EE Y_i:i=1,\dots,k\}$
and

\item
$\RR \theta=\{\RR \theta_{ij}:i,j=1,\dots,k\}$.
\end{itemize}

In \cite[Theorem~5.1]{functor} the second author shows that the right $\OO Y_1$-module \iso\ 
\[
\pre {\OO Y_1}(Y_i\ots \OO Y_1)\otso (Y_j \ots \OO Y_1)_{\OO Y_1}\variso
\pre {\OO Y_1}(Y_j\ots Y_i\ots \OO Y_1)_{\OO Y_1}
\]
defined on elementary tensors by $x_i\ots T\otso \xi\mapsto x_i\ots T\cdot \xi$, where $x_i\in Y_i, T\in  \OO Y_1,$ and $\xi\in  (Y_j \ots \OO Y_1),$ preserves left $\OO Y_1$ module structure. In \cite[Theorem~5.1]{functor} this map was denoted by $U$. Since we will have this \iso\ for any $i,j=1,\dots,k$, in this paper we denote this map by $\nu_{ij}.$ We now define each $\RR \theta_{ij}$ as follows. 

\begin{lem}\label{RU}
For each $i,j$ the map
\[
\RR \theta_{ij}:\EE Y_i\xt_{\OO Y_1} \EE Y_j\variso \EE Y_j\xt_{\OO Y_1} \EE Y_i
\]
defined by
\[
\RR \theta_{ij}=\nu_{ji}\inv (\theta_{ij}\xt 1_{\OO Y_1})\nu_{ij}
\]
is an $\OO Y_1$-\corr\ \iso.
\end{lem}

 We now have a family $\EE Y$ of $\OO Y_1$-\corr s
 and a family $\RR \theta$ of \iso s.

To prove Proposition~\ref{M prop} we must show that $(\EE X,\RR U)$ satisfies the 
symmetry and hexagonal relations.
This requires Lemmas~\ref{M lem 1}--\ref{M lem 2};
the proofs of these lemmas are messy, and
to avoid interrupting the flow we relegate them
 to \apxref{appx}.

 \begin{prop}\label{M prop}
 $(\EE Y,\RR \theta)$ is a \gensys.
 \end{prop}

Now play the game again:
using $Y_2$ instead of $Y_1$,,
we get a \gensys\ $(\EE^2 Y,\RR^2 \theta)$.
We can continue to iterate inductively,
at the $i$th step 
using $Y_i$,
and getting a \gensys\ $(\EE^i Y,\RR^i \theta)$.

We formalize the general step in the iterative process:

\begin{thm}\label{for product}
With the above set-up,
for each $i<k$ and $j=1\dots,k$ 
we have
\begin{equation}
\EE^i Y_j\simeq Y_j\ots \OO(\ast_1^{i}Y_j),
\end{equation}
and, taking $j=e_{i+1}$,
\begin{equation}\label{OEX}
\OO \EE^i Y_{i+1}\simeq \OO(\ast_1^{i+1}Y_j),
\end{equation}
where $\ast_1^{i+1}Y_j$ denotes the \ps\ over $\N^{i+1}$ generated in the obvious way from $Y_1,\dots,Y_{i+1}$.
\end{thm}

The proof of this Proposition involves a number of technical theorems and lemmas, and we devote the next section to them. 

In the special case of \kga s\footnote{And note that this really is a special case --- see for example \cite[Theorem~5.4]{kpq2} for the case $k=2$.}, we finally arrive below at our main result,
which recovers
\cite[Theorem~6.8]{morph}
and
\cite[Theorem~2.6.12]{fletcher}
but via our ``bottom-to-top'' approach\footnote{and with the strong hypothesis of regularity}.
First, if $\Lambda$ is a \kg,  for for each $j\le k$ let
$\Lambda_j$ denote the $j$-graph formed from the set of all $\lambda\in \Lambda$
whose degree has coordinates $d(\lambda)_i=0$ for $i>j$.

\begin{cor}\label{main}
If $\Lambda$ is a regular \kg,
then $\OO \Lambda$ is isomorphic to 
the \cpa\ of 
a \corr\ over
$\OO \Lambda_{k-1}$
formed from $\Lambda$.
\end{cor}

\begin{proof}
By the above, we have
\[
\OO \Lambda\simeq \OO(\ast_1^k Y_j)
\simeq \OO \EE^{k-1}Y_k,
\]
where $\EE^{k-1}Y_k$ is an $\OO \Lambda_{k-1}$-\corr.
\end{proof}

\section{The Structure of 
the Iterated Product  Systems}\label{detail}

In this section we aim to provide a better understanding of the nature of the  \gensys s $(\EE^i Y,\RR^i \theta)$, and to eventually provide a proof of \thmref{for product}. 

\begin{lem}\label{firstiso}
Consider a \ps\ $X$ over $\N^2$, with \gensys\ $(Y,\theta)$.
Then $\OO X\simeq \OO(Y_2\ots \OO Y_1)$.
\end{lem}

\begin{proof}
Write $\EE Y_2=Y_2\ots \OO Y_1$ as usual.
Our strategy is to first construct injective covariant representations $(\sigma_0,\sigma_1)$ and $(\sigma_0,\sigma_2)$
of $Y_1$ and $Y_2$ on $\OO \EE Y_2$,
and then appeal to the Gauge-Invariant Uniqueness theorem (\lemref{gauge-lem}).
Since $A\subseteq \OO {Y_1}\subseteq \OO\EE Y_2$,  we define
\begin{align*}
&\sigma_0(a)=\pi_{\EE {Y_1}}(\pi_{Y_1}(a)).
\\
&\sigma_1(y)=\pi_{\OO{Y_1}}(t_{Y_1}(y))
\end{align*}

Since we don't know whether  $\OO {Y_2}\subseteq \OO \EE Y_2$, defining $\sigma_2: Y_2\to \OO \EE Y_2$ is more challenging. To overcome this challenge, we use the theory of multiplier bimodules recalled in \secref{prelim} as follows: first notice that we may extend the linear map $t_{\EE Y_2}: Y_2\ots \OO Y_1\to \OO\EE Y_2$ to
$\overline{t_{\EE Y_2}}: M\EE Y_2\to M( \OO\EE Y_2).$
Now, define 
\[
\sigma_2(x)=\overline{t_{\EE Y_2}}(x\ots 1_{ M(\OO Y_1)}).
\]
We simplify the notation by writing this as
\[
\sigma_2(x)=\bar{t_{\EE Y_2}}(x\ots 1).
\]
Notice that by the Hewitt--Cohen factorization theorem we can factor $x=y\cdot a$
with $y\in Y_2,a\in A$. Then we have 
\[
x\ots 1=y\cdot a\ots 1= y\ots  \pi_{Y_1}(a),
\]
and thus $x\ots 1\in Y_2\ots \OO Y_1.$ Therefore, we can write $\bar{t_{\EE Y_2}}(x\ots 1)=t_{\EE Y_2}(x\ots 1).$ This means that  the range  of $\sigma_2$ is contained in $\OO  \EE Y_2$, which allows us to define  $\sigma_2: Y_2\to \OO \EE Y_2$ as $\sigma_2(x)=t_{\EE Y_2}(x\ots 1)$.

We claim that the pairs $(\sigma_0, \sigma_1)$ and $(\sigma_0, \sigma_2)$ are indeed representations. We start by checking 
\[
\sigma_0(\< y, y'\>)=\sigma_1(y)^*\sigma_1(y')
\righttext{for}y\in Y_2,
\]
and similarly with $x\in Y_1$.

For $y, y'\in Y_1$ we have 
\[
\sigma_1(y)^*\sigma_1(y')
= \pi_{\OO Y_1}\left(t_{Y_1}(y)^*t_{Y_1}(y')\right)
=  \pi_{\OO Y_1}\left( \pi_{Y_1}(\<y, y'\>_A)\right)
= \sigma_0 (\<y, y'\>_A).
\]

For $x, x'\in Y_2$ we have 
\[
\sigma_2(x)^*\sigma_2(x')
=\pi_{\OO Y_1}\left(\<x\ots 1, x'\ots 1\>_{\OO Y_1}\right)
= \pi_{\OO Y_1}\left(\pi_{Y_1}(\<x,x'\>_A)\right)
=\sigma_0(\<x, x'\>_A).
\]

We next verify the equalities
\[
\sigma_1(a\cdot y)=\sigma_0(a)\sigma_1(y)
\midtext{and} 
\sigma_2(a\cdot x)=\sigma_0(a)\sigma_2(x)
\]
for $a\in A, y\in Y_1$, and $x\in Y_2$:
\begin{align*}
\begin{split}
\sigma_1(a\cdot y)
&=\pi_{\OO Y_1}\bigl(t_{Y_1}(a\cdot y)\bigr)
\\&=\pi_{\OO Y_1}\bigl( \pi_{Y_1}(a)t_{Y_1}(y)\bigr)
\\&=\pi_{\OO {Y_1}}(\pi_{Y_1}(a))\pi_{\OO {Y_1}}(t_{Y_1}(y))
\\&=\sigma_0(a)\sigma_1(y);
\end{split}
\\
\begin{split}
\sigma_2(a\cdot x)
&=t_{\EE Y_2}(a\cdot x\otimes 1)
\\&=\pi_{\OO Y_1}(\pi_{Y_1}(a))t_{\EE Y_2}(x\ots 1)
\\&=\sigma_0(a)\sigma_2(x),
\end{split}
\end{align*}
which concludes the proof of our claim.

We next show that the representations $(\sigma_0,\sigma_i)$ are covariant.
For $(\sigma_0,\sigma_2)$ we need to verify the equality $\sigma_0(a)={\sigma_2}^{(1)}(\varphi_{Y_2}(a))$, for any $a\in A.$ On one hand we have  
\begin{align*}
\sigma_0(a)=\pi_{\OO Y_1}\left(\pi_{Y_1}(a)\right)
&=t_{\EE Y_2}^{(1)}\bigl(\varphi_{\EE Y_2}(\pi_{Y_1}(a))\bigr)& & \text{(since $(\pi_{\EE Y_2}, t_{\EE Y_2})$ is covariant)}\\
&=t_{\EE Y_2}^{(1)}(\varphi_{Y_2}(a)\otimes 1_{\OO Y_1}).
\end{align*}

 Since $Y_2$ is a regular correspondence we have $\varphi_{Y_2}(a)\in \KK(Y_2)$. Without loss of generality assume $\varphi_{Y_2}(a)=\theta_{x,y}$ for some $x,y\in Y_2.$ Recall that the element  $x\ots 1$ of $Y_2\ots  M(\OO Y_1)$ is in fact an element of $\EE Y_2.$ And thus,  the operator $\theta_{x\ots 1,y\ots 1}$ is in $\KK( \EE Y_2).$ In fact, for any $z\in Y_2, S\in \OO Y_1,$ we have  
\begin{align*}
\theta_{x\ots 1,y\ots 1}(z\ots S)
&= (x\ots 1)\< y\ots 1, z\ots S\>_{\OO Y_1}\\
&=x\ots \pi_{Y_1}(\<y,z\>_A)S\\
&= x\<y,z\>_A\ots S\\
&=\theta_{x,y}\otimes 1_{\OO Y_1}(z\ots S),
\end{align*}
and thus $\theta_{x\ots 1,y\ots 1}=\theta_{x,y}\otimes 1_{\OO Y_1}$ as elements  in $\KK(\EE Y_2).$ Now we have 
\begin{align*}
{\sigma_2}^{(1)}(\varphi_{Y_2}(a))&=\sigma_2(x)\sigma_2(y)^*\\
&= t_{\EE Y_2}(x\ots 1)t_{\EE Y_2}(y\ots 1)^*\\
&={t_{\EE Y_2}}^{(1)}(\theta_{x\ots 1,y\ots 1})\\
&= {t_{\EE Y_2}}^{(1)}(\theta_{x,y}\otimes 1_{\OO Y_1})\\
&= {t_{\EE Y_2}}^{(1)}(\varphi_{Y_2}(a)\otimes 1_{\OO Y_1}),
\end{align*}as desired.  Similarly, one can show the covariance of $(\sigma_0, \sigma_1)$ by using the covariance of $(\pi_{Y_1}, t_{Y_1}).$ 

To show that $\sigma_2\sigma_1=\sigma_1\sigma_2\circ \theta_{21}$,
let  $y\in Y_1, x\in Y_2$. On one hand we have
\[
\sigma_2(x)\sigma_1(y)
=t_{\EE Y_2}(x\ots 1)\pi_{\EE Y_2}(t_{Y_1}(y))
= t_{\EE Y_2}(x\ots t_{Y_1}(y)).
\]
On the other hand, let $\theta_{21}(x\ots y)=\lim_{s\to\infty}\sum_{r=1}^{N_s}x_r^s\ots y_r^s$, where $x_r^s\in Y_1, y_r^s\in Y_2$. Then we have 
\begin{align*}
\sigma_1\sigma_2\circ \theta_{21}(x\ots y)&=\lim_{s\to\infty}\sum_{r=1}^{N_s}\sigma_1(x_r^s)\sigma_2(y_r^s)\\
&=\lim_{s\to\infty}\sum_{r=1}^{N_s}\pi_{\OO{Y_1}}(t_{Y_1}(x_r^s))t_{\EE Y_2}(y_r^s\ots 1)\\
&=\lim_{s\to\infty}\sum_{r=1}^{N_s} t_{\EE Y_2}\bigl( t_{Y_1}(x_r^s)\cdot (y_r^s\ots 1)\bigr)\\
&= \lim_{s\to\infty}\sum_{r=1}^{N_s}t_{\EE Y_2}\bigl((1_{Y_2}\otimes V)U_{12}(x_r^s\ots y_r^s)\ots 1\bigr) \\
&= \sum_{r=1}^{N_s}t_{\EE Y_2}\bigl((1_{Y_2}\otimes V)(x\ots\xi\ots 1)\bigr)\\
&=t_{\EE Y_2}(x\ots t_{Y_1}(\xi)).
\end{align*}
Therefore, by \lemref{replem} there is a \rep\ $\psi$ of $X$ in $B$ such that $\psi_{e_i}=\sigma_i$ for $i=1,2$, and $\psi_0=\sigma_0$,
and moreover $\psi$ is essentially unique and \cp\ covariant.

Now that we have a covariant representation $\sigma:X\to \OO \EE Y_2$, in order to prove the \iso\ $\OO X\simeq \OO \EE Y_2$, our last step is to use 
the Gauge-Invariant Uniqueness theorem
(\lemref{gauge-lem}):
first note that the induced \hm\ $\sigma_*:\OO X\to \OO\EE Y_2$ is surjective,
and $\pi_0:A\to \OO\EE Y_2$ is injective.
To finish, we will construct an action $\rho$ of $\T^2$ on $\OO\EE Y_2$
that is compatible with the gauge action on $\OO X$
and as usual we will do our work using the \gensys\ $Y$.

We first claim
that
for any $z\in\T$ the pair $(\gamma_z, 1_{Y_2}\otimes \gamma_z)$ defines a \cst-\corr\ \auto\ of $\pre {\OO Y_1}(\EE Y_2)_{\OO Y_1}$.
To verify this,
we first prove the equality
\[ (1_{Y_2}\otimes \gamma_z)[T\cdot \xi]=\gamma_z(T)\cdot (1_{Y_2}\otimes \gamma_z)(\xi),\]
for any $T\in\OO Y_1, \xi\in \EE Y_2.$ Let $y_1\in Y_1$ and $a\in A.$ It suffices to let $T=t_{Y_1}(y_1)$ and $T=\pi_{Y_1}(a)$, as such elements generate $\OO Y_1.$ Let $y_2\in Y_2$ and $S_n\in (\OO Y_1)^n$. Assume that $U_{1,2}(y_1\ots y_2)=\lim_{n\rightarrow\infty}\sum_{i=1}^{N_n}y_{2,i}^n\ots y_{1,i}^n$, where each $y_{1,i}^n\in Y_1$, and $y_{2,i}^n\in Y_2$. Then we have 
\begin{align*}
\gamma_z(t_{Y_1}(y_1))\cdot (1_{Y_2}\otimes \gamma_z)(y_2\ots S_n)&=z\varphi_{\EE Y_2}(t_{Y_1}(y_1))(y_2\ots z^n S_n)\\
&=z\lim_{n\rightarrow\infty}\sum_{i=1}^{N_n}y_{2,i}^n\ots t_{Y_1}(y_{1,i}^n)z^nS_n \\
&=z^{n+1}\lim_{n\rightarrow\infty}\sum_{i=1}^{N_n}y_{2,i}^n\ots t_{Y_1}(y_{1,i}^n)S_n \\
&=(1_{Y_2}\otimes \gamma_z)\left[t_{Y_1}(y_1)\cdot (y_2\ots S_n)\right],
\end{align*}
and simiarly
\begin{align*}
\gamma_z(\pi_{Y_1}(a))\cdot (1_{Y_2}\otimes\gamma_z)(y_2\ots S_n)&=\pi_{Y_1}(a)\cdot (y_2\ots z^nS_n)\\
&=ay_2\ots z^nS_n\\
&=(1_{Y_2}\otimes \gamma_z)(ay_2\ots S_n)\\
&=(1_{Y_2}\ots\gamma_z)[\pi_{Y_1}(a)\cdot(y_2\ots S_n).]
\end{align*}
We next show the equality,
\[
\<(1_{Y_2}\otimes \gamma_z)(y_2\ots S_n), (1_{Y_2}\otimes \gamma_z)(y_2'\ots S_m)\>=\gamma_z\left(\<y_2\ots S_n, y_2'\ots S_m\>_{\OO Y_1}\right),
\]
for $y_2, y_2'\in Y_2,$ and $S_n\in (\OO Y_1)^n, S_m\in (\OO Y_1)^m$: 
\begin{align*}
&\<(1_{Y_2}\otimes \gamma_z)(y_2\ots S_n),
(1_{Y_2}\otimes \gamma_z)(y_2'\ots S_m)\>_{\OO Y_1}\\
&\qquad=\<y_2\ots z^nS_n, y_2'\ots z^m S_m\>_{\OO Y_1}\\
&\qquad=\<z^nS_n, \<y_2,y_2'\>_A\cdot z^mS_m\>_{\OO Y_1}\\
&\qquad=z^{m-n}S_n^*\pi_{Y_1}(\<y_2,y_2'\>_A)S_m\\
&\qquad=\gamma_z\left(S_n^*\pi_{Y_1}(\<y_2,y_2'\>_A)S_m\right)\\
&\qquad=\gamma_z\left(\<y_2\ots S_n, y_2'\ots S_m\>_{\OO Y_1}\right).
\end{align*}
since surjectivity is straight forward,
we have proven the claim.

We will construct our desired $\T^2$-action $\rho$ on $\OO\EE Y_2$
from commuting actions $\tau^i$ of $\T$.
First, for $z\in \T$ let $\tau_z^1$ denote the automorphism on  $\OO\EE Y_2$ induced from the \cst-correspondence \auto\ $(\gamma_z, 1_{Y_2}\otimes\gamma_z)$ of $\EE Y_2$. Then by Proposition~\ref{comp} we have  
\begin{alignat}{2}
\tau_z^1\circ\pi_{\EE Y_2}&=\pi_{\EE Y_2}\circ\gamma_z\\
\tau_z^1\circ t_{\EE Y_2}&=t_{\EE Y_2}\circ (1_{Y_2}\otimes \gamma_z).
\end{alignat}

Next, consider the map $\beta_z: \EE Y_2\rightarrow \EE Y_2$ defined by $\beta_z(y_2\ots S)=z(y_2\ots S).$ It is straightforward do verify that  the pair $(1_{\OO Y_1},\beta_z) $ defines a \cst-correspondence \auto\ of $\EE Y_2$,  which gives rise to an automorphism $\tau_z^2: \OO\EE Y_2\rightarrow  \OO\EE Y_2$ such that

\begin{alignat}{2}
\tau_z^2\circ\pi_{\EE Y_2}&=\pi_{\EE Y_2}\\
\tau_z^2\circ t_{\EE Y_2}&=t_{\EE Y_2}\circ \beta_z.
\end{alignat}

Now for any $(z,w)\in\T^2$ define an automorphism $\rho_{(z,w)}$ on $\OO\EE Y_2$ by $\rho_{(z,w)}=\tau_w^2\tau_z^1.$ To check the compatibility,
we need
\[
\rho_{(z,w)}\sigma_1(y_1)=z\sigma_1(y_1) \text{ and }
\rho_{(z,w)}\sigma_2(y_2)=w\sigma_2(y_2).
\]

For the first equality:
\begin{align*}
\rho_{(z,w)}\sigma_1(y_1)&=\tau_w^2\tau_z^1\pi_{\EE Y_2}(t_{Y_1}(y_1)) \\
&=\tau_w^2\pi_{\EE Y_2}[z t_{Y_1}(y_1)] & \text{by (0.1)}\\
&=zt_{Y_1}(y_1)=z\sigma_1(y_1) & \text{by (0.3).}
\end{align*}

For the second equality:
\begin{align*}
\rho_{(z,w)}\sigma_2(y_2)&=\tau_w^2\tau_z^1t_{\EE Y_2}(y_2\ots 1)\\
&=\tau_w^2t_{\EE Y_2}(y_2\ots 1) & \text{by (0.2)}\\
&=t_{\EE Y_2}\circ\beta_z(y_2\ots 1) & \text{by (0.4)}\\
&=w\sigma_2(y_2).
\end{align*}
It is clear that the actions $\tau^1$ and $\tau^2$ commute, so $\rho$ gives an action of $\T^2$ on $\OO\EE Y_2$,
and it now follows from the above that $\rho$ is compatible with $\sigma$ and the gauge action on $\OO X$.
\end{proof}

We aim to extend \lemref{firstiso} to a \ps\ consisting of $n$ \corr s for $n\geq 3$. To that we must first 
define a suitable left $\OO(\ast_1^{n-1}Y_j)$-module \hm\
on 
the Hilbert $\OO(\ast_1^{n-1}Y_j)$-module $Y_n\ots \OO(\ast_1^{n-1}Y_j)$.
We need the following remark to accomplish this. 

\begin{rem}\label{universal}
Let $\psi:X\to \OO X$ be the universal covariant representation of a \ps\ $X$
with \gensys\ $(Y,\theta)$.
We can use this to get a \gensys\ $(\sigma,Y,\theta)$ in $\GG^{\OO X}$.
This implies , for $x\in Y_i, y\in Y_j$,  if
\[
\theta_{ij}(x\ots y)=\lim_{s\to\infty}\sum_{r=1}^{N_s} (y_r^s\ots x_r^s), \hspace{1cm} y_r^s\in Y_j, x_r^s\in Y_s
\]
then
\[
\sigma_i(x)\sigma_j(y)
=\lim_{s\to\infty}\sum_{r=1}^{N_s} \sigma_j(y_r^s)\sigma_i(x_r^s).
\]
\end{rem}

\begin{lem}\label{three}
Let $(Y,\theta)$ be a \gensys\ of $n$ $A$-\corr s. Let \mbox{$1<m\leq n$.} Then, for any $k=1,\dots,n$, the Hilbert module $Y_k\ots \OO(\ast_{i\leq m}Y_i)$ is a $C^*$-\corr\ over $\OO(\ast_{i\leq m}Y_i)$.
\end{lem}

\begin{proof}Denote the \ps\  $\ast_{i\leq n}Y_i$ by $X$ and let $\psi$ be the universal covariant representation of $X$. Then as in \remref{universal} we get a generating system $(\sigma,Y,\theta)$ in $\GG^{\OO X}$,
where
\begin{align*}
Y_i&=X_{e_i}\righttext{and}
\sigma_i=\psi_{e_i}.
\end{align*}
For 
$i=1,...,m$
consider the natural absorption \iso s
\[
\V_i: Y_i\ots  \OO X\to \OO X\midtext{determined by}x_i\ots T \mapsto \sigma_{i}(x)T
\]
where  $x\in Y_i$ and $T\in \OO X.$ 
Then
for the left multiplication, define
\[
\Phi_i: Y_i\to \KK(Y_k\ots\OO X)
\]
by
\[
\Phi_i(x)z=(1\xt \V_i)(\theta_{ik}\xt 1_{\OO X})(x\ots z)\righttext{for}z\in Y_k\ots\OO X.
\]
We aim to use this to obtain a covariant representation of $X$ by
appealing to \lemref{replem}.
Notice that since $(\sigma,Y,\theta)$ is a \gensys,
we have
\[
\sigma_i\sigma_j=\sigma_j\sigma_i\circ \theta_{ij}\righttext{for all}i,j\leq n.
\]
To apply \lemref{replem}, we need to show
\begin{equation}\label{Phi}
\Phi_i\Phi_j=\Phi_j\Phi_i\circ \theta_{ij}\righttext{for all}i,j\leq m.
\end{equation}
We first claim that
\begin{equation}\label{Veq}
\V_i(1\xt \V_j)(\theta_{ij}\xt 1)=\V_j(1\xt V_i)\righttext{for all}i,j\leq n.
\end{equation}
Let $x\in Y_i,y\in Y_j$,
and write
\[
\theta_{ij}(x\ots y)=\lim_{s\to\infty}\sum_{r=1}^{N_s}(y_r^s\ots x_r^s).
\]
For $z\in \OO X$ the left-hand side of \eqref{Veq} applied to the elementary tensor $x\ots y\ots z$ is
\begin{align*}
&\lim_s\sum_r \V_i(1\xt \V_j)(\theta_{ij}\xt 1)(x\ots y\ots z)
\\&\quad=\lim_s\sum_r \V_j(1\xt \V_i)(y_r^s\ots x_r^s\ots z)
\\&\quad=\lim_s\sum_r \V_j\bigl(y_r^s\ots \sigma_i(x_r^s)z\bigr)
\\&\quad=\lim_s\sum_r \sigma_j(y_r^s)\sigma_i(x_r^s)z,
\end{align*}
while the right-hand side is
\begin{align*}
\V_j(1\xt \V_i)(x\ots y\ots z)
&=\sigma_i(x)\sigma_j(y)z
\\&=\sigma_i\sigma_j(x\ots y)z
\\&=\sigma_j\sigma_i\circ \theta_{ij}(x\ots y)z
\\&=\lim_s\sum_r \sigma_j\sigma_i(y_r^s\ots x_r^s)z
\\&=\lim_s\sum_r \sigma_j(y_r^s)\sigma_i(x_i^s)z,
\end{align*}
proving the claim.

Turning to \eqref{Phi}, for $x\in Y_i,y\in Y_j,z\in (Y_k\ots \OO X)$ we have
\begin{align*}
&\Phi_i\Phi_j(x\ots y)z
\\&\quad=\Phi_i(x)\Phi_j(y)z
\\&\quad=(1\xt \V_i)\bigl(\theta_{ik}\xt 1)(x\xt \Phi_j(y)z\bigr)
\\&\quad=(1\xt \V_i)\bigl(x\ots (1\xt \V_j)(\theta_{jk}\xt 1)(y\xt z)\bigr)
\\&\quad=(1\xt \V_i)(1\xt 1\xt \V_j)(\theta_{ik}\xt 1\xt 1)(1\xt \theta_{jk}\xt 1)(x\ots y\ots z)
\\&\quad\overset{*}=(1\xt \V_j)(1\xt 1\xt \V_i)(1\xt \theta_{ij}\xt 1)
\\&\quad\hspace{1in}(\theta_{ik}\xt 1\xt 1)(1\xt \theta_{jk}\xt 1)(x\ots y\ots z)
\\&\quad\overset{**}=(1\xt \V_j)(1\xt 1\xt \V_i)(\theta_{jk}\xt 1\xt 1)
\\&\quad\hspace{1in}(1\xt \theta_{ik}\xt 1)(\theta_{ij}\xt 1\xt 1)(x\ots y\ots z)
\\&\quad=\lim_s\sum_r (1\xt V_j)(\theta_{jk}\xt 1)(1\xt 1\xt V_i)(1\xt \theta_{ik}\xt 1)(y_r^s\ots x_r^s\ots z)
\\&\quad=\lim_s\sum_r \Phi_j\Phi_i(y_s^r\ots x_s^r)z
\\&\quad=\Phi_j\Phi_i\circ \theta_{ij}(x\ots y)z,
\end{align*}
where the equality at (*) follows from the claim and the one at (**) follows from the hexagonal relations.
\end{proof}

\begin{cor}\label{maincor}
If $(Y,\theta)$ is a \gensys\ of $n$ $A$-\corr s, then
\[
\OO(\ast_{i\le n}Y_i)\simeq \OO\bigl(Y_n\ots \OO(\ast_{i<n}Y_i)\bigr).
\]
\end{cor}

\begin{proof} 
Let $X=\ast_{i<n}Y_i$,
with universal covariant representation $\psi$,
and let
\[
\omega_i=\psi_{e_i}:Y_i\to \OO X\righttext{for}i<n.
\]
Further let $\DD Y_n=Y_n\ots \OO X$ 
with universal covariant representation $(\Upsilon,T)=(\pi_{\DD Y_n},t_{\DD Y_n})$.
Our strategy is to get a covariant \rep\ 
of $\ast_{i\le n}Y_i$ in $\OO \DD Y_n$,
and we aim to apply \lemref{replem}.
So, we need covariant \rep s
\[
(\sigma_0,\sigma_i):Y_i\to \OO\DD Y_n\righttext{for}i=1,\dots,n.
\]
Since $A\subseteq \OO X \subseteq \OO \DD Y_n$, 
we can define
\[
\sigma_0=\Upsilon\circ \pi_X:A\to \OO \DD Y_n.
\]
We represent the $Y_i$'s differently for $i<n$ and $i=n$:
\[
\sigma_i(x)=
\begin{cases}
\Upsilon\circ \omega_i(x)\case i<n\\
T(x\ots 1)\case i=n.
\end{cases}
\]
Very similarly to \lemref{three}, we can prove that for $i\le n$, $x,y\in Y_i$, and  $a\in A$ we have
\[
\sigma_0(\< x,y\>_A)=\sigma_i(x)^*\sigma_i(y) \midtext{and} 
\sigma_i(a\cdot x)=\pi_0(a)\sigma_i(x).
\] 

Since $(Y,\theta)$ is a \gensys\ in the monoidal category $\GG$ of $A$-\corr s,
to get a \gensys\ in $\GG^{\OO \DD Y_n}_c$ we only
need to check that
\[
\sigma_i\sigma_j=\sigma_j\sigma_i\circ \theta_{ij}\righttext{for all}i,j\le n.
\]
The equality when $i$ or $j=n$ follows exactly as in the proof of \lemref{three}.
For $i,j<n$ we have a \gensys\ $(\omega,Z,\theta)$ in $\GG^{\OO X}_e$
with $Z=\{Y_i\}_{i<n}$.
Since $\Upsilon:\OO X\to \OO \DD Y_n$ is a \cst-\hm, the composition
$(\Upsilon\circ \omega,Z,\theta)$
is a \gensys\ in $\GG^{\OO \DD Y_n}_c$.
Re-inserting $Y_n$, we now have a \gensys\ $(\sigma,Y,\theta)$ in $\GG^{\OO\DD Y_n}_c$, as desired.

It remains to show that the Gauge-Invariant Uniqueness theorem (\lemref{gauge-lem}) is applicable.
This can be accomplished by routinely adjusting the argument for \lemref{firstiso}.
First note that we have a covariant \rep\ $\sigma:Y_1*\cdots*Y_n\to \OO\DD Y_n$,
and the induced \hm\ $\sigma_*:\OO (Y_1*\cdots*Y_n)\to \OO\DD Y_n$
is surjective and $\sigma_0:A\to \OO\DD Y_n$ is injective.
We construct an action $\rho$ of $\T^n$ on $\OO\DD Y_n$
from commuting actions $\tau^i$ of $\T$ for $i=1,\dots,n$.
For $i<n$ we define $\tau^i$ using the same idea as in \lemref{firstiso}
(with the coordinate gauge action $\gamma^i$ of $\T$ associated with $Y_i$),
and similarly for $i=n$
(using the action $\beta$ associated with $Y_n$).
It is again clear that $\tau_i$ for $i=1,\dots,n$ commute,
and hence give an action $\rho$ of $\T^n$,
which is compatible with $\sigma$ and the gauge action on $\OO(Y_1*\cdots*Y_n)$.
\end{proof}

\begin{prop}\label{Prop general iso}
Let $(Y, \theta)$ be a generating system of $n$ $A$-correspondences and let \mbox{$1<k\leq n$.} For any $j=1,\dots,n$, the $\OO(Y_k\ots \OO(\ast_{i<k}Y_i))$-correspondence
\[
Y_j\ots \OO(\ast_{i<k}Y_i)\otimes_{\OO(\ast_{i<k}Y_i)}\OO(Y_k\ots \OO(\ast_{i<k}Y_i))
\]
and the $\OO(\ast_{i\le k}Y_i)$-\corr\ $Y_j\ots \OO(\ast_{i< k}Y_i)$ 
are isomorphic. 
\end{prop}

\begin{proof}
By Corollary~\ref{maincor} there exists an isomorphism
\[
\iota:\OO(\ast_{i\le k}Y_i)\variso \OO(Y_k\ots \OO(\ast_{i<k}Y_i)).
\]
Indeed, letting $\psi: \ast_{i\le k}Y_i\to \OO(\ast_{i\le k}Y_i)$ be the universal covariant representation, it is straightforward to see that $\iota\circ \psi_{e_i}=\sigma_i$ for $i\in\{1,...,k\},$ where each $\sigma_i$ is defined as in Corollary~\ref{maincor}. We will first prove the case $k=2$: Let $\phi:\EE Y_j \otso \OO\EE Y_2 \to  Y_j\ots \OO (Y_1*Y_2)$ be the linear map defined on elementary tensors by 
\[
(x\ots S) \otso T \mapsto x\ots \iota\inv(ST),
\]
for $x\in Y_j, S\in \OO Y_1,$ and $T\in \EE Y_2.$ Then we claim that the pair $(\iota\inv,\phi)$ defines a $C^*$-\corr\ \iso\  
\[
\pre { \OO\EE Y_2}\left(\EE Y_j \otso \OO\EE Y_2\right)_{ \OO\EE Y_2}\variso
\pre {\OO (Y_1*Y_2)} (Y_j\ots \OO (Y_1*Y_2))_{\OO (Y_1*Y_2)}.
\]
Since $\OO \EE Y_2$ is the \csta\ generated by $\sigma_1, \sigma_2$, and $\sigma_0$, it suffices to check the left action for generators $\sigma_1(y_1), \sigma_2(y_2),$ and $\sigma_0(a).$ Let $V_i: Y_i\ots \OO Y_i\to \OO Y_i$ be the natural \iso\ $y_i\ots S\mapsto t_{Y_i}(y_i)S$ for $i\in \{1,2\}.$ Then recall the following:
\begin{enumerate}
\item The left action of $\OO\EE Y_2$ on $\EE Y_j\otso \OO\EE Y_2$ is defined by using the \iso\ $\RR\theta_{2j}: \EE Y_2\otso \EE Y_j\to \EE Y_j\otso \EE Y_2$:  Take $z\in \EE Y_2$, $\xi\in  \EE Y_j,$ and $S\in \OO\EE Y_2.$  The  we have 
\begin{itemize}
\item $t_{\EE Y_2}(z)\cdot (\xi\otso S)
=(1_{\EE Y_j}\otimes V)(\RR\theta_{2j}\otimes 1_{\OO\EE Y_2})(z\otso\xi\otso S).$

\item $\pi_{\EE Y_2}(T)\cdot (\xi\otso S)=T\cdot\xi\otso S,$ for any $T\in \OO Y_1,$ where $V: \EE Y_2\otso  \OO\EE Y_2\to \OO\EE Y_2$  is the natural \iso\ $z\otso S\mapsto t_{\EE Y_2}(z)S.$ 
\end{itemize}

\item For $i\in\{1,2\}$, the left action of $\OO (Y_1*Y_2)$ on $Y_j\ots\OO (Y_1*Y_2)$ is determined by the equalities 
$\psi_{e_i}(y_i)(x\ots K)=(1_{Y_j}\otimes \V_i)[\theta_{ij}(y_i\ots x)\ots K],$ where  $y_i\in Y_i, x\in Y_j,$ and $K\in \OO (Y_1*Y_2)$ (see \lemref{three}). 
\end{enumerate}

Now, let $y_1\in Y_1,$ $x\in Y_j,$ $S\in\OO Y_1$, and $T\in \OO Y_2.$ We prove 
\[
\phi\bigl(\sigma_1(y_1)\cdot(x\ots S\otso T)\bigr)
=\iota\inv(\sigma_1(y_1))\cdot \phi(x\ots S\otso T).
\]

Assume $\theta_{1j}(y_1\ots x)=\lim_{s\to\infty}\sum_{r=1}^{N_s}x_r^s\ots y_r^s$ for $x_r^s\in Y_j, y_r^s\in Y_1.$ Then we have 
\begin{align*}
\sigma_1(y_1)\cdot(x\ots S\otso T)
&=\pi_{\EE Y_2}\left(t_{Y_1}(y_1)\right)\cdot (x\ots S\otso T)
\\
&=t_{Y_1}(y_1)\cdot (x\ots S)\otso T
\\
&= (1_{Y_j}\otimes V_1)\theta_{1j}(y_1\ots x)\ots S\otso T. 
\end{align*}
Therefore, we have 
\begin{align*}
\phi\bigl(\sigma_1(y_1)\cdot(x\ots S\otso T)\bigr)
&=\lim_{s\to\infty}\sum_{r=1}^{N_s} x_r^s\ots\iota\inv\left(t_{Y_1}(y_r^s)\cdot ST\right)
\\
&=\lim_{s\to\infty}\sum_{r=1}^{N_s} x_r^s\ots \iota\inv\left(\pi_{\EE Y_2}(t_{Y_1}(y_r^s))ST\right)
\\
&= \lim_{s\to\infty}\sum_{r=1}^{N_s} x_r^s\ots\iota\inv\left(\sigma_1(y_r^s)ST\right).
\end{align*}

On the other hand we have 
\begin{align*}
\iota\inv(\sigma_1(y_1))\cdot \phi(x\ots S)\otso T
&= \iota\inv(\sigma_1(y_1))\cdot (x\ots i\inv(ST))
\\
&=\psi_1(y_1)\cdot(x\ots \iota\inv(ST))
\\
&=(1_{Y_j}\otimes \V_1)(\theta_{1j}\otimes 1_{\OO(Y_1*Y_2})(y_1\ots x\ots \iota\inv(ST))
\\
&=(1_{Y_j}\otimes \V_1)(\theta_{1j}(y_1\ots x)\ots\iota\inv(ST))
\\
&= \lim_{s\to\infty}\sum_{r=1}^{N_s} x_r^s\ots\psi(y_r^s)\iota\inv(ST)
\\
&=\lim_{s\to\infty}\sum_{r=1}^{N_s} x_r^s\ots \iota\inv(\sigma_1(y_r^s)ST),
\end{align*}
as desired. 

Next we aim to prove
\[
\phi\bigl(\sigma_2(y_2)\cdot [(x\ots S)\otso T]\bigr)
=\sigma_2(y_2)\cdot \phi\bigl((x\ots S)\otso T\bigr)
\righttext{for}y_2\in Y_2.
\]
To that, we first claim the equality
\[
\phi(1_{\EE Y_j}\otimes V)(\nu_{j2}\inv\otimes 1_{\OO\EE Y_2})= (1_{Y_j}\otimes \iota\inv)(1_{Y_j}\otimes V):
\]
Let $y_2\ots S= L\cdot (y'\ots S')$ for some $L\in \OO Y_1, y'\in Y_2, S\in \OO Y_1.$ Then we have 
\begin{align*}
&\phi(1_{\EE Y_j}\otimes V)(\nu_{j2}\inv\otimes 1_{\OO\EE Y_2})(x\ots y_2\ots S\otso T)
\\&\quad=\phi(1_{\EE Y_j}\otimes V)[(x\ots L)\otso (y'\ots S')\otso T]\\
\\&\quad= \phi\bigl(x\ots L \otso t_{\EE Y_2}(y'\ots S')T\bigr)\\
\\&\quad= x\ots \iota\inv\bigl(L\cdot t_{\EE Y_2}(y'\ots S')T\bigr)\\
\\&\quad= x\ots \iota\inv\bigl(t_{\EE Y_2}(y_2\ots S)T\bigr)\\
\\&\quad= (1_{Y_j}\otimes \iota\inv)(1_{Y_j}\otimes V)(x\ots y_2\ots S\otso T),
\end{align*}
completing the proof of our claim. Now we compute 
\begin{align*}
\sigma_2(y_2)\cdot [(x\ots S)\otso T]&=t_{\EE Y_2}(y_2\ots 1)\cdot [(x\ots S)\otso T]
\\
&=(1_{\EE Y_j}\otimes V)(\RR\theta_{2j}\otimes 1_{\OO\EE Y_2})\bigl((y_2\ots 1)\otso (x\ots S)\otso T\bigr)
\\
&=(1_{\EE Y_j}\otimes V)(\nu_{j2}\inv\otimes 1_{\OO\EE Y_2})\lim_{s\to\infty}\sum_{r=1}^{N_s}x_r^s\ots z_r^s\ots S\otso T,
\end{align*}
and our claim gives us 
\begin{align*}
\phi\left(\sigma_2(y_2)\cdot [(x\ots S)\otso T]\right)&=\lim_{s\to\infty}\sum_{r=1}^{N_s}x_r^s\ots \iota\inv\bigl(t_{\EE Y_2}(z_r^s\ots S)T\bigr)\\
&=\lim_{s\to\infty}\sum_{r=1}^{N_s}x_r^s\ots \iota\inv\bigl(t_{\EE Y_2}(z_r^s\ots 1)ST\bigr)\\
&=\lim_{s\to\infty}\sum_{r=1}^{N_s}x_r^s\ots \iota\inv(\sigma_2(z_r^s)ST).
\end{align*}
On the other hand, assuming $\theta_{2j}(y_2\ots x)=\lim_{s\to\infty}\sum_{r=1}^{N_s}x_r^s\ots z_r^s$,
for $x_r^s\in Y_j, z_r^s\in Y_2,$ we have 
\begin{align*}
\iota\inv(\sigma_2(y_2))\cdot(x\ots \iota\inv(ST))&=\psi_2(y_2)\cdot(x\ots \iota\inv(ST))
\\
&=(1_{Y_3}\otimes \V_2)[\theta_{2j}(y_2\ots x)\ots \iota\inv(ST)]
\\
&=\lim_{s\to\infty}\sum_{r=1}^{N_s}x_r^s\ots \psi_2(z_r^s)\iota\inv(ST)
\\
&=\lim_{s\to\infty}\sum_{r=1}^{N_s}x_r^s\ots \iota\inv(\sigma_2(z_r^s)ST),
\end{align*}
as desired. 
We leave it to the reader to prove 
\[
\phi\left(\sigma_0(a)\cdot\bigl((x\ots S\otso T)\bigr)\right)=\iota\inv(\sigma_0(a))\cdot \phi(x\ots S\otso T),
\]
for any $a\in A$, as it is straight forward. 

Now that we proved our Proposition for $k=2$, we may identify $\EE^2 Y_j$ with the $C^*$-\corr\ $Y_j\ots \OO (Y_1*Y_2)$ over $\OO (Y_1*Y_2)$, and we will restate  the \iso\ 
$\RR^2\theta_{3j}: \EE^2Y_3\otimes_{\OO(Y_1*Y_2)}\EE^2 Y_j\to \EE^2 Y_j\otimes_{\OO(Y_1*Y_2)}\EE^2 Y_3,$ accordingly.
Following the same strategy as in the case $k=2$, one can now prove the Proposition for the case $k=3$, and identify $\EE^3 Y_j$ with the $C^*$-\corr\  $Y_j\ots \OO (Y_1*Y_2*Y_3)$ over $\OO (Y_1*Y_2*Y_3)$. This iterative process will give us the desired result. 
\end{proof}

Now we are ready to describe the general picture of our iterative process: \begin{itemize}
    \item 
    Start with 
    a \gensys\ $(Y,\theta)$ of $k$ $C^*$-correspondences.
    Then we have
    \[
    \theta_{1i}: Y_1\ots Y_i\to Y_i\ots Y_1
    \]
    for any $i=1,\dots,k$.
    This \iso\  leads to the construction of the $C^*$-\corr\ $\EE Y_i= \pre {\OO Y_1}(Y_i\ots \OO Y_1)_{\OO Y_1}$. Now, by using $\theta_{ij}$ we define $C^*$-\corr\ \iso s 
\[
\RR \theta_{ij}: \EE Y_i \otimes_{\OO Y_1} \EE Y_j \to \EE Y_j \otimes_{\OO Y_1} \EE Y_i,
\] giving us the \gensys\ $(\EE Y, \RR \theta)$ over $\N^k.$ 

\item 
By using the \iso\
\[
\RR \theta_{2i}: \EE Y_2 \otimes_{\OO Y_1} \EE Y_i \to \EE Y_i \otimes_{\OO Y_1} \EE Y_2,
\]
get the $C^*$-\corr\
\[
\EE^2Y_i= (Y_i\ots \OO Y_1)\otimes_{\OO Y_1} \OO\EE Y_2
\]
over  $\OO (Y_2\ots \OO Y_1),$  which is isomorphic to the $C^*$-\corr\  $\left(Y_i\ots \OO (Y_1* Y_2)\right)$ over $\OO (Y_1*Y_2)$.
This allows us to identify $\EE^2Y_i$ as
\[
\EE^2Y_i : = \pre {\OO (Y_1*Y_2)}(Y_i \ots \OO (Y_1* Y_2))_{\OO (Y_1*Y_2)}.
\]
Now, exactly as in the first step we define isomorphisms
\[
\RR^2 \theta_{ij}: \EE^2 Y_i\otimes_{\OO (Y_1*Y_2)}\EE^2Y_j\to
\EE^2Y_j \otimes_{\OO (Y_1*Y_2)}\EE^2 Y_i,
\]
and get a \gensys\ $(\EE^2Y, \RR^2 \theta)$.

\item
Then $\RR^2 \theta_{3i}$ will give us
\[
\EE^3Y_i
=\pre{\OO\EE^2Y_3}\left(
\EE^2Y_i\xt_{\OO(Y_1*Y_2)}\OO\EE^2Y_3
\right)_{\OO\EE^2Y_3}.
\]
\end{itemize}

Repeating this procedure will lead us to our main result \thmref{for product}.


\appendix
\section{Proofs of Lemmas}\label{appx}

In this appendix we record the proofs of some technical lemmas used in \secref{product}. 

\begin{proof}[Proof of \lemref{RU}]
It suffices to show that for all $i,j$ the $A-\OO Y_1$ \corr\ \iso\
\[
\theta_{ij}\otimes 1_{\OO Y_1}: Y_i\ots Y_j\ots \OO Y_1 \to Y_j\ots Y_i\ots \OO Y_1
\]
preserves the left $\OO Y_1$ module structure. Let $(\pi_{Y_1}, t_{Y_1})$ be the universal covariant representation of $Y_1$. Let $\phi:= \theta_{ij}\xt 1_{\OO Y_1}$ to ease the notation. It suffices to show the equalities 
\[ \phi\bigl( \pi_{Y_1}(a)\cdot (\xi\ots S)\bigr) = \pi_{Y_1}(a)\cdot \phi(\xi\ots S)\]
and 
\[ \phi\bigl( t_{Y_1}(x)\cdot(\xi\ots S)\bigr) = t_{Y_1}(x)\cdot \phi(\xi\ots S)\]
for any $x\in Y_1, a\in A, \xi\in Y_i\ots Y_j$, and $S\in \OO Y_1. $

The first equality can be verified easily so we will only prove the second one. Let \mbox{$V:Y_1\ots \OO Y_1\to \OO Y_1$} be the natural \iso\ defined  $x\ots S\mapsto t_{Y_1}(x)S$. By the construction of the left action of  $\OO Y_1$ on the Hilbert $\OO Y_1$-module $(Y_i\xt_A Y_j \xt_A \OO Y_1)$ we have 
\[ t(x)\cdot (\xi\ots S)= (1_{Y_i\ots Y_j}\xt V)(1_{Y_i}\xt \theta_{1t}\xt 1_{\OO X_1})(\theta_{1s}\xt 1_{Y_j}\xt 1_{\OO Y_1})(x\ots \xi\ots S). \] And, by the construction of the left action of  $\OO Y_1$ on the Hilbert $\OO Y_1$-module $(Y_j\xt_A Y_i \xt_A \OO Y_1)$ we have 
\[ t(x)\cdot \phi(\xi\ots S) = (1_{Y_j\ots Y_i}\xt V)(1_{Y_j}\xt \theta_{1s}\xt 1_{\OO Y_1})(\theta_{1t}\xt 1_{Y_i}\xt 1_{\OO Y_1})(x\ots \phi(\xi\ots S)). \]

Therefore, recalling that $\phi= \theta_{ij}\xt 1_{\OO Y_1}$, the equation we want to prove is 
\begin{equation}
\begin{split}\label{Eq1}
\hspace{-0.5cm} (\theta_{ij}\xt 1_{\OO Y_1})(1_{Y_i\ots Y_j}\xt V)(1_{Y_i}\xt \theta_{1t}\xt 1_{\OO Y_1})(\theta_{1s}\xt 1_{Y_j}\xt 1_{\OO Y_1})\\
&\hspace{-10cm}= (1_{Y_j\ots Y_i}\xt V)(1_{Y_j}\xt \theta_{1s}\xt 1_{\OO Y_1})(\theta_{1t}\xt 1_{Y_i}\xt 1_{\OO Y_1})(1_{Y_1}\xt \theta_{ij}\xt 1_{\OO Y_1}).  
\end{split}
\end{equation} 
By the hexagonal relations
we have 
\[
(1_{Y_i} \xt \theta_{1t})(\theta_{1s}\xt 1_{Y_j})
= (\theta_{ij}\inv\xt 1_{Y_1})(1_{Y_j}\xt \theta_{1s})
(\theta_{1t}\xt 1_{Y_i})(1_{Y_1}\xt \theta_{ij}).
\]
Now, the left hand side of 
\eqref{Eq1} becomes
\begin{align*}
&(\theta_{ij}\xt 1_{\OO Y_1})(1_{Y_i\ots Y_j}\xt V)(\theta_{ij}\inv\xt 1_{Y_1\ots \OO Y_1})\\&\hspace{1in}(1_{Y_j}\xt \theta_{1s}\xt 1_{\OO Y_1})(\theta_{1t}\xt 1_{Y_i\ots \OO Y_1})(1_{Y_1}\xt \theta_{ij}\xt 1_{\OO Y_1}),
\end{align*}
which is exactly the right hand side of equation \eqref{Eq1}. 
\end{proof}

In order to prove Proposition~\ref{M prop} we need Lemmas \ref{M lem 1} and \ref{M lem 2} below.
Recall 
that for any $i,j=1,\dots,k$ we have the \iso\ 
\[
\nu_{ij}: \pre {\OO Y_1}(Y_i\ots \OO Y_1)\otso (Y_j \ots \OO Y_1)_{\OO Y_1}\to \pre {\OO Y_1}(Y_i\ots Y_j\ots \OO Y_1)_{\OO Y_1}
\]
defined on elementary tensors by $x_i\ots T\otso \xi\mapsto x_i\ots T\cdot \xi$, where $x_i\in Y_i, T\in  \OO Y_1,$ and $\xi\in (Y_j \ots \OO Y_1)$.
Similarly, we denote by $\nu_{\ell(i\xt j)}$ the isomorphism
\[\pre {\OO Y_1}(Y_\ell\ots \OO Y_1)\otso (Y_i \ots Y_j\ots \OO Y_1)_{\OO Y_1}\simeq \pre {\OO Y_1}(Y_\ell\ots Y_i\ots Y_j\ots  \OO Y_1)_{\OO Y_1},\]
and by $\nu_{(\ell\xt i)j}$ the isomorphism
\[\pre {\OO Y_1}(Y_\ell\ots Y_i\ots  \OO Y_1)\otso (Y_j \ots \OO Y_1)_{\OO Y_1}\simeq \pre {\OO Y_1}(Y_\ell\ots Y_i\ots Y_j\ots \OO Y_1)_{\OO Y_1}.\]

\begin{lem}\label{M lem 1}
For any $\ell,i,j=1,\dots,k$ we have the following properties. 
\begin{enumerate}[label=\textup{(\arabic*)}]
\item $(\theta_{\ell j}\xt 1_{Y_i\ots \OO_{Y_1}})(1_{Y_\ell}\xt \nu_{ji})
=\nu_{(j\xt \ell)i}(\theta_{\ell j}\xt 1_{\OO_{Y_1}}\xt 1_{Y_i\ots \OO_{Y_1}})$.

\item $\nu_{\ell(i\xt j)}(1_{Y_\ell\ots \OO Y_1}\xt \theta_{ji}\xt 1_{\OO Y_1})= (1_{Y_\ell}\xt \theta_{ji}\xt 1_{\OO Y_1})\nu_{\ell(j\xt i)}$.

\item $ (1_{Y_\ell}\xt \nu_{ji})(\nu_{\ell j}\xt 1_{Y_i\ots \OO Y_1})(1_{Y_\ell\ots \OO Y_1}\xt\nu_{ji}\inv)=\nu_{\ell(j\xt i)}$.
\end{enumerate}
\end{lem}

\begin{proof}
It suffices to check elementary tensors.
Let $S, T\in \OO Y_1$, $x_\ell\in Y_\ell $, $x_i\in Y_i\ $, and $x_j\in Y_j$.
For (1) we have 
\begin{align*}
&(\theta_{\ell j}\xt 1_{Y_i\ots \OO_{Y_1}})(1_{Y_\ell }\xt \nu_{ts})\left(x_\ell\ots (x_j\ots S)\otso (x_i\ots T)\right)\\
&\quad=(\theta_{\ell j}\xt 1_{Y_i\ots \OO_{Y_1}})\left(x_\ell\ots x_j \ots S\cdot(x_i\ots T)\right)\\
&\quad= \theta_{\ell j}(x_\ell\ots x_j)\ots S\cdot (x_i\ots T)\\
&\quad= \nu_{(j\xt \ell)i}\bigl( \theta_{\ell j}(x_\ell\ots x_j)\ots S)\otso (x_i\ots T)\bigr)\\
&\quad= \nu_{(j\xt \ell)i}(\theta_{\ell j}\xt 1_{\OO_{Y_1}}\xt 1_{Y_i\ots \OO_{Y_1}})
\left(x_\ell\ots (x_j\ots S)\otso (x_i\ots T)\right).
\end{align*}

For (2) we have 
\begin{align*}
&\nu_{\ell(i\xt j)}(1_{Y_\ell \ots \OO Y_1}\xt \theta_{ji}\xt 1_{\OO Y_1})\bigl((x_\ell\ots S)\otso (x_j\ots x_i\ots T)\bigr)\\
&\quad=x_\ell\ots S\cdot\bigl( \theta_{ji}(x_j\ots x_i)\xt T\bigr)\\
&\quad=x_\ell\ots (\theta_{ji}\xt 1_{\OO Y_1})\bigl( S\cdot (x_j\ots x_i\ots T)\bigr)\\
&\quad= (1_{Y_\ell }\xt \theta_{ji}\xt 1_{\OO Y_1})\nu_{\ell(j\xt i)}\bigl((x_\ell\ots S)\otso (x_j\ots x_i\ots T)\bigr).
\end{align*}

For (3) we have 
\begin{align*}
&(1_{Y_\ell }\xt \nu_{ji})(\nu_{\ell j}\xt 1_{Y_i\ots\OO Y_1})(1_{Y_\ell \ots\OO Y_1}\xt\nu_{ji}\inv)\left((x_\ell\ots S) \otso( x_j\ots x_i\ots T)\right)\\
&\quad= (1_{Y_\ell }\xt \nu_{ji})x_\ell\ots S\cdot \bigl(\nu_{ji}\inv(x_j\ots x_i\ots T)\bigr)\\
&\quad= (1_{Y_\ell }\xt \nu_{ji})x_\ell\ots \nu_{ji}\inv\left(S\cdot (x_j\ots x_i\ots T)\right)\\
&\quad=x_\ell\ots S\cdot (x_\ell\ots x_i\ots T)\\
&\quad= \nu_{\ell(j\xt i)}\left((x_\ell\ots S)\otso (x_j\ots x_i\ots T)\right).
\end{align*}
\end{proof}

\begin{lem}\label{M lem 2}
For any $\ell,i,j=1,\dots,k$ we have the following properties.
\begin{align}
\label{one}
&(\RR \theta_{\ell j}\xt 1_{Y_i\ots \OO Y_1})(1_{Y_\ell\ots \OO Y_1}\xt \RR \theta_{ij})
\\&\quad\notag=\xi\inv(\theta_{\ell j}\xt 1_{Y_i}\xt 1_{\OO Y_1})(1_{Y_\ell}\xt \theta_{ij}\xt 1_{\OO Y_1})\nu;
\\
\label{two}
&(1_{Y_j\ots \OO Y_1}\xt \RR \theta_{i\ell})(\RR \theta_{ij}\xt 1_{Y_\ell\ots \OO Y_1})(1_{Y_i\ots \OO Y_1}\xt \RR \theta_{\ell j})(\RR \theta_{\ell i}\xt 1_{Y_j\ots \OO Y_1})
\\&\quad\notag=\xi\inv\bigl((1_{Y_j}\xt \theta_{i\ell})(\theta_{ij}\xt 1_{Y_\ell})(1_{Y_i}\xt \theta_{\ell j})(\theta_{\ell i}\xt 1_{Y_j})\xt 1_{\OO Y_1}\bigr)\nu,
\end{align}
where $\xi:=\nu_{\ell(i\xt j)}(1_{Y_\ell\ots \OO Y_1} \xt \nu_{ij})$ is the isomorphism
\begin{align*}
&\pre {\OO Y_1} (Y_\ell\ots \OO Y_1)\otso (Y_i\ots \OO Y_1)
\otso (Y_j\ots \OO Y_1)_{\OO Y_1}
\\&\hspace{.5in}\variso \pre {\OO Y_1} (Y_\ell\ots Y_i\ots Y_j\ots \OO Y_1)_{\OO Y_1},
\end{align*}
and $\nu:=\nu_{(j\xt \ell)i}(\nu_{j\ell}\xt 1_{Y_i\ots \OO Y_1}) $ is the \iso\  
\begin{align*}
&\pre {\OO Y_1} (Y_j\ots \OO Y_1)\otso (Y_\ell\ots \OO Y_1)\otso (Y_i\ots \OO Y_1)_{\OO Y_1}
\\&\hspace{.5in}\variso \pre {\OO Y_1} (Y_j\ots Y_\ell\ots Y_i\ots \OO Y_1)_{\OO Y_1}.
\end{align*}
\end{lem}

\begin{proof}
For (A.2), it suffices to verify the equality
\[
\begin{split}
&\nu_{(j\xt \ell )i}(\theta_{\ell j}\xt 1_{\OO Y_1}\xt 1_{Y_i\ots \OO Y_1})(\nu_{\ell j}\xt 1_{Y_i\ots \OO Y_1})(1_{Y_\ell \ots \OO Y_1}\xt \nu_{ji}\inv)
\\&\quad=(\theta_{\ell j}\xt 1_{Y_i\ots \OO Y_1})(1_{Y_\ell }\xt \theta_{ij}\xt 1_{\OO Y_1})\nu_{\ell (i\xt j)}(1_{Y_\ell \ots \OO Y_1}\xt \theta_{ji}\xt 1_{\OO Y_1}).
\end{split}
\]
We have 
\begin{align*}
&\nu_{(j\xt \ell )i}(\theta_{\ell j}\xt 1_{\OO Y_1}\xt 1_{Y_i\ots \OO Y_1})(\nu_{\ell j}\xt 1_{Y_i\ots \OO Y_1})(1_{Y_\ell \ots \OO Y_1}\xt \nu_{ji}\inv)
\\&\quad=(\theta_{\ell j}\xt 1_{Y_i\ots \OO Y_1})(1_{Y_\ell }\xt \nu_{ji})(\nu_{\ell j}\xt 1_{Y_i\ots \OO Y_1})(1_{Y_\ell \ots \OO Y_1}\xt \nu_{ji}\inv)
\righttext{(by (1) of \lemref{M lem 1})}
\\&\quad= (\theta_{\ell j}\xt 1_{Y_i\ots \OO Y_1})\nu_{\ell (j\xt i)}
\righttext{(by (3) of \lemref{M lem 1})}
\\&\quad= (\theta_{\ell j}\xt 1_{Y_i\ots \OO Y_1})(1_{Y_\ell }\xt \theta_{ij}\xt 1_{\OO Y_1})\nu_{\ell (i\xt j)}(1_{Y_\ell \ots \OO Y_1}\xt \theta_{ji}\xt 1_{\OO Y_1})
\\&\hspace{2in}\righttext{(by (2) of \lemref{M lem 1})}
\end{align*}

(A.3) can be shown by using the first item of this Lemma, and items (2)--(3) of \lemref{M lem 1}. 
\end{proof}

\begin{proof}[Proof of Proposition~\ref{M prop}]
We want to show that $(\EE Y, \RR \theta)$ is a \gensys, where $\EE Y$ is a set of $C^*$-correspondences $\{ \EE Y_i = Y_i\ots\OO Y_1\}_{i=1,\dots,k}$ over $\OO Y_1,$ and  $\RR \theta$ is a set of \corr\ \iso s
\[
\RR \theta_{ij}:\EE Y_i\xt_{\OO Y_1}\EE Y_j
\variso  \EE Y_j\xt_{\OO Y_1}\EE Y_i.
\]
Verifying the symmetry condition is straightforward since the \iso s $\theta_{ij}$ satisfy the symmetry condition: 
\[
\RR \theta_{ij}\inv=\nu_{ij}\inv(\theta_{ji}\otimes 1_{\OO Y_1})\nu_{ji}=\RR \theta_{ji}.
\]
We verify the \emph{hexagonal relations}:
for all 
$\ell,i,j=1,\dots,k$, we will prove the equality  
\begin{align*}
&(1_{\EE Y_j}\xt \RR \theta_{i\ell})(\RR \theta_{ij}\xt 1_{\EE Y_r})(1_{\EE Y_i}\xt \RR \theta_{\ell j})(\RR \theta_{rs}\xt 1_{\EE Y_j})
\\&\hspace{1in}=(\RR \theta_{\ell j}\xt 1_{\EE Y_i})(1_{\EE Y_r}\xt \RR \theta_{ij}).
\end{align*}
Notice that on both sides of the equation we have $\OO Y_1$-\corr\ \iso s from
\[
\pre {\OO Y_1} (Y_r\ots \OO Y_1)\otso 
(Y_i\ots \OO Y_1)\otso
(Y_j\ots \OO Y_1)_{\OO Y_1}
\]
to
\[\pre {\OO Y_1} (Y_j\ots \OO Y_1)\otso
(Y_r\ots \OO Y_1)\otso
(Y_i\ots \OO Y_1)_{\OO Y_1}
\]
as the hexagonal relation requires. We have 
\begin{align*}
& (1_{\EE Y_j} \xt \RR \theta_{i\ell})(\RR \theta_{ij}\xt 1_{\EE Y_r})(1_{\EE Y_i}\xt \RR \theta_{\ell j})(\RR \theta_{rs}\xt 1_{\EE Y_j})
\\&\quad=\xi\inv\bigl((1_{Y_j}\xt \theta_{i\ell})(\theta_{ij}\xt 1_{Y_r})(1_{Y_i}\xt \theta_{\ell j})(\theta_{rs}\xt 1_{Y_j})\xt 1_{\OO Y_1}\bigr)\nu
\\&\hspace{.5in}
\righttext{(by \eqref{two})}
\\&\quad=\xi\inv\bigl((\theta_{\ell j}\xt 1_{Y_i})(1_{Y_r}\xt \theta_{ij})\xt 1_{\OO Y_1}\bigr) \nu
\\&\hspace{.5in}\righttext{(by the hexagonal relation on $\{\theta_{rs}\}$)}
\\&\quad=(\RR \theta_{\ell j}\xt 1_{\EE Y_i})(1_{\EE Y_r}\xt \RR \theta_{ij})
\righttext{(by \eqref{one}).}
\end{align*}
\end{proof}



\providecommand{\bysame}{\leavevmode\hbox to3em{\hrulefill}\thinspace}
\providecommand{\MR}{\relax\ifhmode\unskip\space\fi MR }
\providecommand{\MRhref}[2]{%
  \href{http://www.ams.org/mathscinet-getitem?mr=#1}{#2}
}
\providecommand{\href}[2]{#2}

\end{document}